\documentclass[12pt,english,fleqn,liststotoc,bibtotoc,idxtotoc,BCOR7.5mm,tablecaptionabove]{amsart}
\usepackage[T1]{fontenc}
\usepackage[latin1]{inputenc}
\usepackage{geometry}
\usepackage{color}
\usepackage{amsthm}
\usepackage{amstext}
\usepackage{amssymb}
\usepackage{amsmath}
\usepackage{graphicx}
\usepackage{young}
\usepackage{xypic}
\usepackage{amscd}
\usepackage[section]{placeins}
\usepackage[svgnames]{xcolor}
\usepackage[unicode=true,
bookmarks=true,bookmarksnumbered=true,bookmarksopen=false,
breaklinks=false,pdfborder={0 0 1},backref=false,colorlinks=true]
{hyperref}
\hypersetup{pdftitle={},
	pdfauthor={Nick Early},
	pdfsubject={},
	pdfkeywords={},
	linkcolor=black, citecolor=black, urlcolor=blue, filecolor=blue,  pdfpagelayout=OneColumn, pdfnewwindow=true,  pdfstartview=XYZ, plainpages=false, pdfpagelabels}
\usepackage{breakurl}

\makeatletter

\theoremstyle{plain}
\newtheorem{thm}{\protect\theoremname}
\theoremstyle{plain}
\newtheorem{conjecture}[thm]{\protect\conjecturename}
\theoremstyle{plain}

\theoremstyle{remark}
\newtheorem{claim}[thm]{\protect\claimname}
\theoremstyle{plain}
\newtheorem{lem}[thm]{\protect\lemmaname}
\theoremstyle{plain}
\newtheorem{prop}[thm]{\protect\propositionname}
\theoremstyle{remark}

\theoremstyle{remark}
\newtheorem{rem}[thm]{\protect\remarkname}
\theoremstyle{definition}
\newtheorem{defn}[thm]{\protect\definitionname}
\theoremstyle{definition}
\newtheorem{example}[thm]{\protect\examplename}
\theoremstyle{plain}
\newtheorem{cor}[thm]{\protect\corollaryname}
\theoremstyle{plain}

\numberwithin{thm}{section}

\usepackage{ifpdf}
\ifpdf

\IfFileExists{lmodern.sty}
{\usepackage{lmodern}}{}

\fi 

\usepackage{multicol}

\makeatother

\providecommand{\claimname}{\inputencoding{latin9}Claim}
\providecommand{\conjecturename}{\inputencoding{latin9}Conjecture}
\providecommand{\corollaryname}{\inputencoding{latin9}Corollary}
\providecommand{\definitionname}{\inputencoding{latin9}Definition}
\providecommand{\examplename}{\inputencoding{latin9}Example}
\providecommand{\lemmaname}{\inputencoding{latin9}Lemma}
\providecommand{\notename}{\inputencoding{latin9}Note}
\providecommand{\propositionname}{\inputencoding{latin9}Proposition}
\providecommand{\questionname}{\inputencoding{latin9}Question}
\providecommand{\remarkname}{\inputencoding{latin9}Remark}
\providecommand{\theoremname}{\inputencoding{latin9}Theorem}
\providecommand{\problemname}{\inputencoding{latin9}Problem}

\newcommand\twoheaduparrow{\mathrel{\rotatebox{90}{$\twoheaduparrow$}}}
\newcommand\twoheaddownarrow{\mathrel{\rotatebox{270}{$\twoheaddownarrow$}}}

\newenvironment{nouppercase}{%
	\renewcommand{\uppercasenonmath}[1]{}}{}

\begin{document}
	\author{Nick Early}
	\thanks{Perimeter Institute for Theoretical Physics \\
		email: \href{mailto:earlnick@gmail.com}{earlnick@gmail.com}}
	
	\title{W\MakeLowercase{eighted blade arrangements and the positive tropical} G\MakeLowercase{rassmannian}}
	
	\begin{nouppercase}
	\maketitle
\end{nouppercase}
	\begin{abstract}
		In this paper, we continue our study of blade arrangements and the positroidal subdivisions which are induced by them on hypersimplices $\Delta_{k,n}$.  The prototypical blade is a tropical hypersurface which is generated by a system of $n$ affine simple roots of type $SL_n$ and as such enjoys a cyclic symmetry.  When placed at the center of a simplex, a blade induces a decomposition into $n$ maximal cells which are combinatorially cubes, known as Pitman-Stanley polytopes.
		
		We introduce a complex $(\mathfrak{B}_{k,n},\partial)$ of weighted blade arrangements, and we prove that the positive tropical Grassmannian surjects onto the top component of the complex, such that the induced weights on blades in the faces $\Delta_{2,n-(k-2)}$ of $\Delta_{k,n}$ are (1) nonnegative and (2) their support is weakly separated.
		
		We introduce a hierarchy of elementary weighted blade arrangements for all hypersimplices which is minimally closed under the boundary maps $\partial$, and we conjecture that any such element of this hierarchy induces a ray of the positive tropical Grassmannian $\text{Trop}^+G(k,n)$.  We apply our results to classify up to isomorphism type all rays of the positive tropical Grassmannian $\text{Trop}_+ G(3,n)$ for $n\le 9$.  Along the way, we prove linear independence for a certain set of planar kinematic invariants introduced previously.  

	\end{abstract}
	
	\begingroup
	\let\cleardoublepage\relax
	\let\clearpage\relax
	\tableofcontents
	\endgroup

\section{Introduction}
In this paper, we show that the positive tropical Grassmannian $\text{Trop}^+G(k,n)$, introduced in \cite{SpeyerWilliams2003}, embeds as a subfan of the space of weighted arrangements of cyclically twisted tropical hyperplanes, called blades \cite{Early19WeakSeparationMatroidSubdivision}, on the vertices of the hypersimplex $\Delta_{k,n}$.  We introduce a basis for the space of height functions over $\Delta_{k,n}$, which define surfaces whose lower-envelope project down onto $\Delta_{k,n}$ to induce a particular kind of subdivision, called a multi-split, see for instance \cite{HermannSplits} and in particular \cite{Schroeter}.  This induces a basis for the space of kinematic functions, see Theorem \ref{thm: planar basis}.

We give a combinatorial characterization of the image using positivity and certain pairwise orthogonality constraints on the second hypersimplicial faces of $\Delta_{k,n}$, each of which is modulo translation equal to $\Delta_{2,n-(k-2)}$.  We formulate a conjecture which would construct many new rays of the positive tropical Grassmannian $\text{Trop}^+G(k,n)$ and would have deep implications for the study of the singularities certain rational functions which occur in theoretical physics in the study of scattering amplitudes, the generalized biadjoint scalar amplitude, introduced by Cachazo, Early, Guevara and Mizera (CEGM).  The conjecture is that the elements which we construct induce positroidal subdivisions of $\Delta_{k,n}$ which are coarsest, and generate rays of $\text{Trop}^+G(k,n)$.

The tropical Grassmannian $\text{Trop}G(k,n)$, introduced by Speyer-Sturmfels in \cite{SpeyerStumfelsTropGrass}, parametrizes tropicalizations of linear spaces, while the Dressian parametrizes all tropical linear spaces.  These two have so-called positive analogs (for the totally positive Grassmannian see \cite{SpeyerWilliams2003}) which were independently shown to coincide in \cite{arkani2020positive,SpeyerWilliams2020}.

One of the first and main challenges in the study of $\text{Trop}^+G(k,n)$ is a complete description of its one-skeleton, when it is given the structure of a polyhedral complex.  We single out certain distinguished collections of weighted blade arrangements which are to serve as building blocks for the 1-skeleton of $\text{Trop}^+G(k,n)$; we formulate a conjecture.  We conjecture that any such weighted blade arrangement induces a positroidal subdivision of $\Delta_{k,n}$ which is coarsest, and thus a ray of $\text{Trop}^+G(k,n)$.  

Blades were defined by Ocneanu in \cite{OcneanuVideo} and were first studied in the context of combinatorial geometry and physics in \cite{EarlyBlades}.  A blade is a cyclically symmetric tropical hypersurface constructed from a system of affine roots of type $SL_n$; it is is linearly isomorphic to a tropical hyperplane and there is one (nondegenerate) blade for each cyclic order on $\{1,\ldots, n\}$.  But blades can be degenerated and translated, and one can take weighted arrangements or formal linear combinations.  In \cite{Early19WeakSeparationMatroidSubdivision} it was shown that certain matroidal blade arrangements on the vertices of the hypersimplex $\Delta_{k,n}$ are in bijection with weakly separated collections of $k$ element subsets $\{12,\ldots, n\}$.  Weakly separated collections, in turn, have also played a central role in the study of zonotopal tilings and the KP equation, see for instance \cite{Galashin,GalashinPostnikovWilliams}.

One of the novelties of embedding the positive tropical Grassmannian into the space of matroidal weighted blade arrangements is the following: rather than taking a discrete set of heights over the vertices of a polytope and projecting down the lower envelope of the convex hull of the heights in the standard way, one realizes matroid subdivisions as certain weighted collections of polyhedral cones, subject to the positivity and mutual orthogonality constraints which we prove in this paper in the case of positroids.  This approach automatically removes the lineality space, for regular subdivisions.  In sum, our approach suggests the possibility of extending our results to the context of polypositroids \cite{LamPostnikov2018}, and beyond, to generalized permutohedra.

Indeed, taking weighted arrangements of blades of all cyclic orders one will encounter novel linear relations of quasi-shuffle type among their indicator functions \cite{EarlyBlades}.  On the other hand, one could study arrangements of multiple copies of a blade with a fixed cyclic order $(1,2,\ldots, n)$.  This study was initiated in \cite{Early19WeakSeparationMatroidSubdivision}, where it was shown that the set of blade arrangements on the vertices of a hypersimplex $\Delta_{k,n}$ which induce matroidal (in fact positroidal) subdivisions, is in bijection with weakly separated collections of $k$-element subsets of $\{1,\ldots, n\}$.  We generalize the result from \cite{Early19WeakSeparationMatroidSubdivision}: we show that the set of matroidal \textit{weighted} blade arrangements is isomorphic to the positive tropical Grassmannian (modulo the lineality subspace).

Prior to this, in \cite{CEGM2019} CEGM discovered and generalized to higher dimensional projective spaces the Cachazo-He-Yuan (CHY) formalism \cite{CHY2014} for a particular tree-level scattering amplitude, the biadjoint scalar $m^{(2)}(\alpha,\beta)$ for $n$ cycles $\alpha,\beta$, and they discovered moreover a connection to the tropical Grassmannian.  The stucture of the CEGM generalized biadjoint scalar scattering amplitude \cite{CEGM2019} $m^{(k)}(\alpha,\alpha)$ has subsequently been computed symbolically and numerically with a variety of related methods: using certain collections and arrays of metric trees, called generalized Feynman diagrams \cite{BC2019,CGUZ2019}; using cluster algebra mutations to map the set of maximal cones of the (nonnegative) tropical Grassmannian \cite{Drummond2019A,Drummond2019B,Drummond2020,HenkePapa2019}; using matroid subdivisions and matroidal blade arrangements \cite{Early19WeakSeparationMatroidSubdivision,Early2019PlanarBasis}; Minkowski sums of Newton polytopes of a positive parametrization of the nonnegative Grassmannian \cite{AHL2019Stringy}; codimension 1 limit configurations of points in the moduli space $X(k,n)$ \cite{HeRenZhang2020}; and by direct tabulation of compatible sets of maximal cells in (regular) positroidal subdivisions of $\Delta_{k,n}$, in \cite{LPW2020}.  In \cite{SoftTheoremtropGrassmannian} a soft factorization theorem for $m^{(k)}(\alpha,\alpha)$ was proved for $k\ge 3$.

Evidence was given in \cite{CEGM2019} that one can access more of the tropical Grassmannian, beyond the positive tropical Grassmannian, by generalizing the integrand beyond the usual cyclic product of $k\times k$ minors known as the generalize Parke-Taylor factor, see Example \ref{example: nonplanar subdivision} and Appendix \ref{subsec:generalizedBiadjointScalar}.  The conjecture was that when $\alpha = \beta$, then $m^{(k)}(\alpha,\alpha)$ is a sum of rational functions which are in bijection with the maximal cones in $\text{Trop}^+ G(k,n)$.  In \cite{BC2019,CGUZ2019} these were calculated explicitly by taking the Laplace transform of the space of generalized Feynman diagrams, that is collections of metric trees for $m^{(3)}(\alpha,\alpha)$ and then arrays of metric trees for $m^{(4)}(\alpha,\alpha)$, subject to certain compatibility conditions on the metrics.  The study of collections of metric trees was initiated in \cite{Hermann How to draw}.

Our main technical result is Lemma \ref{lem:posDressToBlades}, which shows that the space of weighted matroidal blade arrangements $\mathcal{Z}_{k,n}$ on the vertices of the hypersimplex $\Delta_{k,n}$ is characterized by the positive tropical Plucker relations.  From this we deduce our main result in Theorem \ref{thm:tropGrass onto Zkn} that the positive tropical Grassmannian maps onto $\mathcal{Z}_{k,n}$ with fiber the $n$-dimensional so-called lineality subspace.  One novel feature of $\mathcal{Z}_{k,n}$ is the boundary map; in this way, weighted blade arrangements are forced to satisfy linear relations compatibly with the face poset of the hypersimplex.  In other words, $\mathcal{Z}_{k,n}$ embeds into the top component of a graded complex $\mathfrak{B}^\bullet_{k,n}$, such that on the faces of $\Delta_{k,n}$ one has separate positive tropical Grassmannians which are glued together by the boundary map!

In the concluding Section \ref{sec:building blocks} we go beyond general theory to introduce a hierarchy of elementary building blocks for all $\mathcal{Z}_{k,n}$ which is minimally closed with respect to the boundary maps $\partial: \mathfrak{B}^\ell_{k,n}\rightarrow \mathfrak{B}^{\ell+1}_{k,n}$.  We conjecture that every such weighted blade arrangement induces a coarsest positroidal subdivision of $\Delta_{k,n}$.  Here the boundary maps $\partial$ take weighted blade arrangements on faces of $\Delta_{k,n}$ of codimension $\ell$ to a sum of weighted blade arrangements on faces of codimension $\ell+1$ of $\Delta_{k,n}$.

In the Appendix, we recall the construction of $m^{(k)}(\alpha,\beta)$, culminating in the computation in Section \ref{subsec:generalizedBiadjointScalar} of several of the $n=6$ point amplitudes $m^{(3)}(\alpha,\beta)$.  Here one of the poles $-\eta_{246} +\eta_{124}+\eta_{256}+\eta_{346}$ appearing in $m^{(3)}((1,2,\ldots, 6),(1,2,5,6,3,4))$ can easily be recognized to belong to the family in Section \ref{sec:building blocks}.  The whole expression is then recognized as a (collection of) coarsest weighted blade arrangements which generate one of the bipyramids in $\text{Trop}_+ G(3,6)$.

\newpage

\section{Blades and positroidal subdivisions}\label{sec:Permutohedral blades and alcoved polytopes}
Let $\mathcal{H}_{k,n}$ be the affine hyperplane in $\mathbb{R}^n$ where $\sum_{i=1}^n x_i=k$.  For integers $1\le k\le n-1$, denote by $\Delta_{k,n} = \left\{x\in \lbrack 0,1\rbrack^n:\sum_{j=1}^n x_j=k \right\}$ the $k^\text{th}$ hypersimplex of dimension $n-1$.  For a subset $J\subseteq \{1,\ldots, n\}$, denote $x_J = \sum_{j\in J} x_j$, and similarly for basis vectors, $e_J = \sum_{j\in J} e_j$.  Denote $J_i = \{i,i+1,\ldots, i+k-1\}$.  Then $e_{J_1},\ldots, e_{J_n}$ are the $n$ \textit{frozen} vertices of $\Delta_{k,n}$.  Call a vertex $e_J\in\Delta_{k,n}$ \textit{totally nonfrozen} if the set $J$ partitions $\{1,\ldots, n\}$ into exactly $k$ cyclic intervals.

For any subset $L\in\binom{\lbrack n\rbrack}{m}$ with $1\le m\le k$, define the face
$$\partial_L(\Delta_{k,n}) = \left\{x\in\Delta_{k,n}:  x_\ell = 1  \text{ for all } \ell\in L\right\}.$$
Up to translation, this is $\Delta_{k-\vert L\vert,n-\vert L\vert}$.  Denote by $\binom{\lbrack n\rbrack}{k}$ be the set of all $k$-element subsets of $\{1,\ldots, n\}$.

In \cite{OcneanuVideo}, A. Ocneanu introduced plates and blades, as follows.
\begin{defn}[\cite{OcneanuVideo}]\label{defn:blade}
	A decorated ordered set partition $((S_1)_{s_1},\ldots, (S_\ell)_{s_\ell})$ of $(\{1,\ldots, n\},k)$ is an ordered set partition $(S_1,\ldots, S_\ell)$ of $\{1,\ldots, n\}$ together with an ordered list of integers $(s_1,\ldots, s_\ell)$ with $\sum_{j=1}^\ell s_j=k$.  It is said to be of (hypersimplicial) type $\Delta_{k,n}$ if we have additionally $1\le s_j\le\vert S_j\vert-1 $, for each $j=1,\ldots, \ell$.  In this case we write $((S_1)_{s_1},\ldots, (S_\ell)_{s_\ell}) \in \text{OSP}(\Delta_{k,n})$, and we denote by $\lbrack (S_1)_{s_1},\ldots, (S_\ell)_{s_\ell}\rbrack$ the convex polyhedral cone in $\mathcal{H}_{k,n}$, that is cut out by the facet inequalities
	\begin{eqnarray}\label{eq:hypersimplexPlate}
	x_{S_1} & \ge & s_1 \nonumber\\
	x_{S_1\cup S_2} & \ge & s_1+s_2\nonumber\\
	& \vdots & \\
	x_{S_1\cup\cdots \cup S_{\ell-1}} & \ge & s_1+\cdots +s_{\ell-1}.\nonumber
	\end{eqnarray}
	These cones were studied as \textit{plates} by Ocneanu.
	Finally, the \textit{blade} $(((S_1)_{s_1},\ldots, (S_\ell)_{s_\ell}))$ is the union of the codimension one faces of the complete simplicial fan formed by the $\ell$ cyclic block rotations of $\lbrack (S_1)_{s_1},\ldots(S_\ell)_{s_\ell},\rbrack$, that is
	\begin{eqnarray}\label{eq: defn blade}
	(((S_1)_{s_1},\ldots, (S_\ell)_{s_\ell})) = \bigcup_{j=1}^\ell \partial\left(\lbrack (S_j)_{s_j},(S_{j+1})_{s_{j+1}},\ldots, (S_{j-1})_{s_{j-1}}\rbrack\right).
	\end{eqnarray}
	Here $(((S_1)_{s_1},\ldots, (S_\ell)_{s_\ell}))$ is a \textit{hypersimplicial} blade of type $\Delta_{k,n}$ if  $$((S_1)_{s_1},\ldots, (S_\ell)_{s_\ell})\in\text{OSP}(\Delta_{k,n}).$$

\end{defn}

In what follows, let us denote for convenience the cone $\Pi_j = \lbrack j,j+1,\ldots, j-1\rbrack$, that is 
$$\Pi_j = \left\{t_1(e_1-e_2)+t_2(e_2-e_3)+\cdots + \widehat{t_j(e_j-e_{j+1})}+\cdots +t_n(e_n-e_1):t_1,\ldots, t_n\ge 0\right\}.$$
It is easy to check that the $n$ closed simplicial cones $\Pi_1,\ldots, \Pi_n$ form a complete simplicial fan, centered at the origin, in the hyperplane $\mathcal{H}_{0,n}$ in $\mathbb{R}^n$ where $\sum_{i=1}^n x_i=0$.  For an argument that uses the Minkowski algebra of polyhedral cones, see \cite{EarlyBlades}.

One finds that for any $x\in \mathcal{H}_{0,n}$, there is a unique subset $\{i_1,\ldots, i_m\} \in \binom{\lbrack n\rbrack}{m}$ with $m\ge 1$ such that $x$ is in the relative interior of $\Pi_{i_1}\cap \cdots \cap \Pi_{i_t}$.

\begin{rem}
	When there is no risk of confusion, depending on the context we shall use the notation $\lbrack (S_1)_{s_1},\ldots, (S_\ell)_{s_\ell}\rbrack$ for the cone in $\mathcal{H}_{k,n}$ or for the matroid polytope obtained by intersecting it with the hypersimplex $\Delta_{k,n}$.
\end{rem}

Let $\beta = ((1,2,\ldots, n))$ be the standard blade; as noted in \cite{Early19WeakSeparationMatroidSubdivision}, this is isomorphic to a tropical hyperplane.

Any point $v\in \mathbb{R}^n$ gives rise to a translation $\beta_v$ of $\beta$ by the vector $v$.  When $v = e_J$ is a vertex of a hypersimplex $\Delta_{k,n}$, then we write simply $\beta_J$.
\begin{rem}
	In what follows, starting in Section \ref{sec: hypersimplicial blade hypersimplex}, we shall also denote by $\beta_J$ the elements of a vector space $\mathfrak{B}_{k,n}$.
\end{rem}

\begin{defn}
	A matroid polytope $P$ is a subpolytope of a hypersimplex $\Delta_{k,n}$ such that every every edge of $P$ is parallel to an edge of $\Delta_{k,n}$, i.e. it is in a root direction $e_i-e_j$. A matroid polytope such that every facet is defined an equation of the form $x_i+x_{i+1}+\cdots +x_j \ge r_{ij}$ is called a positroid polytope.  Here the interval is understood to be cyclic modulo $n$.
\end{defn}

Let $y_i = x_1+\cdots +x_i$.  Following \cite{AlcovedPolytopes}, the set of affine hyperplanes of the form $y_{j}-y_{i} = x_{i+1}+x_{i+2}+\cdots +x_j = r_{ij}$ in $\mathcal{H}_{k,n}$, for positive integers $r_{ij}$, induces a triangulation of $\Delta_{k,n}$ into the Eulerian number $\mathcal{A}_{k-1,n-1}$ simplices, called alcoves.
\begin{defn}[\cite{AlcovedPolytopes}]
	A polytope in $\mathbb{R}^{n-1}$ is said to be \textit{alcoved} if its facet inequalities are of the form $b_{ij}\le y_i-y_j \le  c_{ij}$ for some collection of integer parameters $b_{ij}$ and $c_{ij}$.
\end{defn}

As noted in \cite{AlcovedPolytopes}, any alcoved polytope comes with a natural triangulation into Weyl alcoves.

\begin{defn}
	A matroid subdivision is a decomposition $\Pi_1\sqcup\cdots \sqcup \Pi_d$ of a hypersimplex $\Delta_{k,n}$ such that each pair of maximal cells $\Pi_i,\Pi_j$ intersects only on their common face, and such that each $\Pi_i$ is a matroid polytope.  
	
	A matroid subdivision is called \textit{positroidal} if every maximal cell $\Pi_i$ has its facets given by equations $x_{\alpha_{i+1}}+x_{\alpha_{i+2}}+\cdots+x_{\alpha_{i+m}}=r_{i,i+m}$ for some integers $r_{i,i+m}$, where $i\in\{1,\ldots, n\}$ and $1\le m \le n-2$, and some given cyclic order $\alpha=(\alpha_1,\alpha_2,\ldots, \alpha_n)$.  When $\alpha=(1,2,\ldots, n)$ then the subdivision is positroidal.
\end{defn}
Let $\mathcal{C} = \{C_1,\ldots, C_\ell\}$ and $\mathcal{D} = \{D_1,\ldots, D_m\}$ be two matroid subdivisions of $\Delta_{k,n}$; then $\mathcal{C}$ \textit{refines} $\mathcal{D}$ if every maximal cell $D_i$ of $\mathcal{D}$ is a union of maximal cells $C_j$ of $\mathcal{C}$.  Similarly, $\mathcal{C}$ \textit{coarsens} $\mathcal{D}$ if every maximal cell $C_i$ of $\mathcal{C}$ is a union of maximal cells $D_j$ of $\mathcal{D}$.

According to the standard construction, the set of matroid subdivisions form a poset with respect to refinement. 
\begin{defn}\label{defn:multisplit}
Let $d\ge 2$.  A $d$-split of an $m$-dimensional polytope $P$ is a coarsest subdivision $P = P_1\cup\cdots \cup P_d$ into $m$-dimensional polytopes $P_i$, such that the polytopes $P_i$ intersect only on their common faces, and such that 
$$\text{codim}(P_1\cap\cdots\cap P_d) = d-1.$$  If $d$ is not specified, then we shall use the term multi-split.  
\end{defn}

Recall that the Eulerian number $A_{k,n}$ is the number of permutations of $\{1,\ldots, n-1\}$ having $k-1$ descents.

\begin{thm}[\cite{Ocneanu2013}]\label{thm:decoratedOSPs Eulerian number}
	There is a bijection between decorated ordered set partitions of hypersimplicial type $\Delta_{k,n}$ and derangements\footnote{derangements are fixed-point free permutations.} of $\{1,\ldots, n\}$ with $k$ ascents and $n-k$ descents.  
	
	The number of decorated ordered set partitions $((S_1)_{s_1},\ldots, (S_\ell)_{s_\ell}) \in \text{OSP}(\Delta_{k,n})$ such that $1\in S_1$ is the Eulerian number $A_{k-1,n-1}$.
\end{thm}
In Corollary \ref{cor:hypersimplicial blades count} we are not including the trivial blade $((12\cdots k)) = \Delta_{k,n}$ (which is in fact 0 in the vector space $\mathcal{B}_{k,n}$, introduced in Section \ref{sec: hypersimplex blade complex}), which induces the trivial subdivision of $\Delta_{k,n}$.  If we include that then the number jumps by one, to the Eulerian number exactly.
\begin{cor}\label{cor:hypersimplicial blades count}
	There are exactly $A_{k-1,n-1}-1$ hypersimplicial blades of type $\Delta_{k,n}$.
\end{cor}
As a second immediate Corollary we have the following enumeration of multi-split matroidal subdivisions of $\Delta_{k,n}$.
\begin{cor}
	There are $A_{k-1,n-1}-1$ (nontrivial) multi-split matroidal subdivisions.  
\end{cor}
\begin{proof}
	This follows by combining Theorem \ref{thm:decoratedOSPs Eulerian number} with a result from \cite{Early19WeakSeparationMatroidSubdivision}, where it was shown that the multi-split subdivisions of $\Delta_{k,n}$ are exactly those subdivisions induced by hypersimplicial blades $(((S_1)_{s_1},\ldots, (S_\ell)_{s_\ell}))$.  But hypersimplicial blades are by construction in bijection with decorated ordered set partitions $((S_1)_{s_1},\ldots, (S_\ell)_{s_\ell})\in \text{OSP}(\Delta_{k,n})$ that have at least two blocks, modulo cyclic block rotation.
\end{proof}

Thus, when the (trivial) 1-split subdivision induced by the blade $((12\cdots n_k)) = \Delta_{k,n}$ is included, then there are exactly $A_{k-1,n-1}$ multi-split matroidal subdivisions of $\Delta_{k,n}$.

Let us recall some definitions and results from \cite{Early19WeakSeparationMatroidSubdivision}.

\begin{defn}
	A \textit{blade arrangement} is a superposition of a number of copies of $((1,2,\ldots, n))$ on the vertices of a given hypersimplex $\Delta_{k,n}$, where $2\le k\le n-2$.  A \textit{weighted} blade arrangement is a linear combination (often, but not always, with integer coefficients) of blades $\beta_J$.
\end{defn}
In the case that all numbers in a linear combination are nonnegative, then any weighted blade arrangement maps to a unique blade arrangement, obtained by setting all coefficient weights to 1.

\begin{thm}[\cite{Early19WeakSeparationMatroidSubdivision}]\label{thm: multisplit blade}
	Let $e_I$ be a vertex of $\Delta_{k,n}$ and fix a cyclic order, say without loss of generality $\alpha = (1,2,\ldots, n)$.  Then, the translated blade $((1,2,\ldots, n))_{e_{I}}$ induces a multi-split matroid subdivision of $\Delta_{k,n}$, with $\ell$ maximal cells, separated by a hypersimplicial blade,
	$$\left(((1,2,\ldots, n))_{e_{I}}\right)\cap \Delta_{k,n} = \left(((S_1)_{s_1},(S_2)_{s_2},\ldots, (S_\ell)_{s_\ell})\right)\cap \Delta_{k,n},$$
	where $((S_1)_{s_1},\ldots, (S_\ell)_{s_\ell})\in \text{OSP}(\Delta_{k,n})$ is determined by $\alpha$ and $e_I$, satisfies the property that $\ell$ equals the number of cyclic intervals in $I$.  In particular, the blade induces the trivial matroid subdivision, if and only if $I$ is a cyclic interval.
\end{thm}
\begin{defn}[\cite{LeclercZelevinsky}]\label{defn:weakly separated}
	Let $I,J\in\binom{\lbrack n\rbrack}{k}$ be given.  
	
	The subsets $I,J$ are \textit{weakly separated} if they satisfy the property that no four elements $i_1,i_2,j_1,j_2$ with $i_1, i_2\in (I\setminus J)$ and $j_1,j_2\in (J\setminus  I)$ have
	$$i_1<j_1<i_2<j_2$$
	or one of its cyclic rotations.
	
	If subsets $J_1,\ldots, J_m\in\binom{\lbrack n\rbrack}{k}$ are pairwise weakly separated, then $\mathcal{C} = \{J_1,\ldots, J_m\}$ is called a \textit{weakly separated collection}.
\end{defn}
In the usual geometric interpretation for $k$-element subsets, c.f. \cite{WeakSeparationPostnikov}, $I$ and $J$ are weakly separated if there exists a chord separating the sets $I\setminus J$ and $J\setminus I$ when drawn on a circle.  Identifying each $k$-element subset $J$ of $\{1,\ldots, n\}$ with the vertex $e_J$ gives rise to a notion of weak separation for arrangements of vertices of the form $\{e_{I_1},\ldots, e_{I_m}\}\subset \Delta_{k,n}$.

\begin{thm}[\cite{Early19WeakSeparationMatroidSubdivision}]\label{thm: weakly separated matroid subdivision blade}
	Given a collection of vertices $e_{I_1},e_{I_2},\ldots, e_{I_m}\in\Delta_{k,n}$, the subdivision of $\Delta_{k,n}$ that is induced by the blade arrangement 
	$$\{((1,2,\ldots, n))_{e_{I_1}},((1,2,\ldots, n))_{e_{I_2}},\ldots, ((1,2,\ldots, n))_{e_{I_m}}\}$$
	is matroidal (in fact positroidal) if and only if $\{I_1,\ldots, I_m\}$ is weakly separated.
\end{thm}

When a blade arrangement induces a positroidal subdivision on $\Delta_{k,n}$, call it matroidal.  

Similarly, call a \textit{weighted} blade arrangement matroidal when it induces on every second hypersimplicial face a positroidal subdivision.  See also Definition \ref{defn:Xkn YKn ZKn} and Corollary \ref{cor: regular positroidal subdivision blades}.   Our main result in Theorem \ref{thm:tropGrass onto Zkn} is to show that the set of matroidal weighted blade arrangements is isomorphic to the positive tropical Grassmannian modulo the n-dimensional lineality subspace.

Let us consider what may happen when weights are allowed to be negative; this is a new phenomenon first occurring for $k=3$ hypersimplices.  Negative coefficients in a weighted blade arrangement are possible so long as on each boundary copy of $\Delta_{2,n-(k-2)}$ in $\Delta_{k,n}$ all coefficients become nonnegative.  This happens for the first time for the following weighted blade arrangement on $\Delta_{3,6}$:
$$-\beta_{2,4,6} + \beta_{1,2,4}+\beta_{3,4,6}+\beta_{2,5,6},$$
where as in the sequel we use the notation $\beta_{J} = ((1,2,\ldots, n))_{e_J}$.  This can be seen to induce a 2-split on each of the six facets of $\Delta_{3,6}$, where the negative terms completely cancel and we get
$$\beta^{(1)}_{2,4}+\beta^{(2)}_{1,4}+\beta^{(3)}_{4,6} + \beta^{(4)}_{3,6} + \beta^{(5)}_{2,6} + \beta^{(6)}_{2,5}.$$
For instance, $\beta^{(6)}_{2,5}$ induces the positroidal subdivision of the face $\left\{x\in\Delta_{3,6}: x_6=1 \right\} \simeq \Delta_{2,5}$ with the two maximal cells separated by the (hypersimplicial) blade
$$((12_1 345_1))^{(6)} = \left\{x\in \mathbb{R}^6: x_6=1,\ x_1+x_2 = x_3+x_4+x_5 =1\right\}.$$

\section{Hypersimplicial Vertex Space $\mathbb{R}^{\binom{n}{k}}$ and Kinematic Space}\label{sec: height functions and kinematic space}
Let $\left\{e^J: J \in \binom{\lbrack n\rbrack}{k} \right\}$ be the standard basis for $\mathbb{R}^{\binom{n}{k}}$.

\begin{defn}
	The \textit{kinematic space} for the hypersimplex $\Delta_{k,n}$ is the subspace $\mathcal{K}_{k,n}$ of $\mathbb{R}^{\binom{n}{k}}$ defined by
\begin{eqnarray}\label{eq: kinematic space}
\mathcal{K}_{k,n} = \left\{(s_J)\in\mathbb{R}^{\binom{n}{k}}: \sum_{J\in\binom{\lbrack n\rbrack}{k}:\ J\ni a} s_{J}=0 \text{ for each } a=1,\ldots, n\right\}.
\end{eqnarray}
\end{defn}
According to the standard construction in combinatorial geometry, see for instance \cite{TriangulationsBook}, any height function over the vertices of a polyhedron $P$ defines a continuous, piecewise linear surface over $P$, which in turn induces a \textit{regular} subdivision of $P$, obtained by projecting down onto $P$ the folds in the surface.  When $P=\Delta_{k,n}$ the height function takes values in $\mathbb{R}^{\binom{n}{k}}$; as noted in the proof of Proposition \ref{prop:frozen constant slope}, the relations cutting out the kinematic space as a subspace of $\mathbb{R}^{\binom{n}{k}}$ vanish exactly on the space of continuous, piecewise-linear surfaces that have constant slope over the whole hypersimplex\footnote{Such spaces have been studied in the work of Lafforgue, for instance \cite{Lafforgue}.}.  

In this paper we are concerned with a particular subset of the kinematic space which is defined by restricting to height functions that induce regular subdivisions where all maximal cells are a particular kind of matroid polytope, such that each octahedral face in say 
$$\left\{x\in \lbrack 0,1\rbrack^n: x_a+x_b+x_c+x_d=2\right\}\cap \Delta_{k,n}$$
is subdivided in a way that is compatible with a corresponding cyclic order inherited from a given cyclic order $\alpha=(1,2,\ldots, n)$.  Namely, over each octahedron the folds of the surface should project onto either $x_a+x_b = x_c+x_d=1$ or $x_a+x_d=x_b+x_c=1$ where $a<b<c<d$ cyclically.  Such regular subdivisions are of course exactly the positroidal subdivisions\footnote{For results related to this characterization see \cite{arkani2020positive,Early19WeakSeparationMatroidSubdivision,LPW2020,SpeyerWilliams2020}.}.

Let us now describe how to construct the configuration space of such height functions; to this end we introduce a distinguished set of height functions which form a basis of linear functions on $\mathcal{K}_{k,n}$, where each induces elements of a family of matroid subdivisions, the positroidal \textit{multi-splits}, in terms of which any regular positroidal subdivision has a particularly convenient expansion.

Define a piecewise-linear function $\mathfrak{h}(x)$ on $\mathbb{R}^n$ by
$$h(x) = \min\{L_1(x),\ldots, L_n(x)\},$$
where
$$L_j = x_{j+1} + 2x_{j+2}+\cdots (n-1)x_{j-1}.$$
We shall restrict its domain to the hyperplane $\mathcal{H}_{0,n}$ where $x_1+\cdots +x_n=0$.
\subsection{Bases for $\mathbb{R}^{\binom{n}{k}}$}

Recall the notation $\Pi_j = \lbrack j,j+1,\ldots, j-1\rbrack$, that is 
$$\Pi_j = \left\{t_1(e_1-e_2)+t_2(e_2-e_3)+\cdots +t_n(e_n-e_1):t_1,\ldots, t_n\ge 0;\ t_j=0\right\}.$$
One can easily check that $h$ is linear on each $\Pi_j$.  In particular we have Proposition \ref{prop: linear A}.
\begin{prop}\label{prop: linear A}
	If $x\in \Pi_j$, then $\min\{L_1(x),\ldots, L_n(x)\} = L_j(x)$, hence
	$$h(x) = L_j(x).$$
\end{prop}

Translating $h$ to the vertices of hypersimplex $\Delta_{k,n}$ gives rise to a collection of piecewise-linear functions $\rho_J(x) = h(x-e_J)$ for $x\in\Delta_{k,n}$, and restricting these to the vertices of $\Delta_{k,n}$ determines an (integer-valued) height function, which we shall encode by a vector in $\mathbb{Z}^{\binom{n}{k}}$.  

To summarize, let $\rho_J:\Delta_{k,n}\rightarrow \mathbb{R}$ be the translation of $h$ by vector $e_J$:
\begin{eqnarray}\label{eq:rho}
	\rho_J(x) = h(x-e_J).
\end{eqnarray}

Now put
\begin{eqnarray}\label{eqn:height function}
	\mathfrak{h}_{J}:=\sum_{e_I\in\Delta_{k,n}} \rho_J(e_I)e^I \in\mathbb{R}^{\binom{n}{k}}
\end{eqnarray}

\begin{prop}\label{prop:expansion h}
	Given a lattice point $x\in (\mathcal{H}_{0,n}\cap \mathbb{Z}^n)$, then there exist unique integers $t_{i_1},\ldots, t_{i_\ell}> 0$, such that 
	$$x = \sum_{j\in \left\{i_1,\ldots, i_\ell\right\}^c} t_j(e_j-e_{j+1}).$$
\end{prop}

\begin{proof}
	Given a point $x\in (\mathcal{H}_{0,n}\cap \mathbb{Z}^n)$ as above, then it is in the relative interior of some (maximal) intersection $\Pi_{i_1}\cap \cdots \cap \Pi_{i_\ell}$.  Taking $t_{i_1} = \cdots =t_{i_\ell}=0$ then $x$ is in the relative interior of a simplicial cone of dimension $n-\ell$, so we have a unique expansion 
	$$x = \sum_{j\in \left\{i_1,\ldots, i_\ell\right\}^c} t_j(e_j-e_{j+1})$$
	with integers $t_j>0$.
\end{proof}
We may now define for any pair of vertices $e_I,e_J\in \Delta_{k,n}$, an integer
$$d(e_I,e_J)= \sum_{j=1}^{n}t_j,$$ 
where we set $x = e_J-e_I$ in Proposition \ref{prop:expansion h}, noting that $t_j = 0$ for all $j\in \left\{i_1,\ldots, i_\ell\right\}$.  

Then $ d(e_I,e_J)$ is the smallest (positive) number of steps required to walk from $e_I$ to $e_J$, where each step has to be in one of the root directions $e_1-e_2,e_2-e_3,\ldots, e_n-e_1$.

Define a linear operator $\mathcal{L}:\mathbb{R}^{\binom{n}{k}} \rightarrow \mathbb{R}^{\binom{n}{k}}$ by extending by linearity the map
$$\mathcal{L}(e^J) = \sum_{M\in C_J}(-1)^{1+(M\cdot M) }e^{J_M},$$
where $C_J$ is the cube
$$C_J=\left\{J_M=\{j_1-m_1,\ldots, j_t-m_t\}: M=(m_1,\ldots, m_t) \in \{0,1\}^t\right\},$$
with $\{j_1,\ldots, j_t\}$ being the cyclic initial points of the cyclic intervals of $J$. 

Further let $\mathcal{R}:\mathbb{R}^{\binom{n}{k}}\rightarrow \mathbb{R}^{\binom{n}{k}}$ be the linear operator induced by extending by linearity the assignment
$$e^J \mapsto -\frac{1}{n}\sum_{I\in\binom{\lbrack n\rbrack}{k}} \rho_I(e_{J})e^I.$$

\begin{lem}\label{lem: distance walk}
	For any pair of vertices $e_I,e_J$ of $\Delta_{k,n}$, we have
	$$h(e_I-e_J) = -d(e_I,e_J).$$
	Then,
	$$\mathfrak{h}_J  = \frac{1}{n}\sum_{e_I\in \Delta_{k,n}} d(e_I,e_J)\mathcal{L}(e^I).$$
\end{lem}
We prove the result for a blade translated to an arbitrary lattice point $v\in\mathbb{Z}^n$ and then specialize to the case when $v$ is a vertex of the hypersimplex $\Delta_{k,n}$.
\begin{proof}
	Given $u,v \in \mathbb{Z}^n$ with $\sum_{i=1}^n u_i =\sum_{i=1}^n v_i \in\mathbb{Z}$, as in the statement of the Theorem, then by Proposition \ref{prop:expansion h}, $u-v$ expands uniquely as
	$$u-v = \sum_{j=1}^{n} t_j (e_j-e_{j+1}),$$
	for integers $t_1,\ldots, t_n\ge 0$, where at least one of the $t_i$ is zero.  Supposing without loss of generality that $t_n=0$, then $h(u-v)$ evaluates to $L_1(x)$.
	
	But
	\begin{eqnarray*}
		L_1(u-v) & = & (u_2-v_2)+2(u_{3}-v_3)+ \cdots + (n-1)(u_n-v_n)\\
		& = & (u-v)\cdot \sum_{j=2}^{n}\left(e_j+e_{j+1}+\cdots +e_{n}\right)\\
		& = & \left(\sum_{j=1}^{n-1} t_j (e_j-e_{j+1})\right)\cdot \sum_{j=2}^{n}\left(e_j+e_{j+1}+\cdots +e_{n}\right)\\
		& = & -\sum_{j=1}^{n-1}t_j,
	\end{eqnarray*}
	which equals $-d(u,v)$.  As $u-v$ is by assumption in the domain of linearity of $L_1$, it follows that 
	$$h(u-v) = L_1(u,v) =-d(u,v).$$
	
	In particular, when $(u,v) = (e_I,e_J)$ is a pair of vertices of $\Delta_{k,n}$ the first result holds; the statement about $\mathfrak{h}_{J}$ follows immediately from the definition.

\end{proof}

\begin{prop}\label{cor:inversionFormula}

	We have
	
	\begin{eqnarray}\label{eq: inversion cubical blades}
	\sum_{L\in C_J}(-1)^{(1+L\cdot L) }\mathcal{R}(e^{J_L}) = e^J
	\end{eqnarray}
	where $L\cdot L$ is the number of 1's in the 0/1 vector $L$.  Moreover, 
	$$\sum_{I}\rho_J(e_I)\mathcal{L}(e^I) = e^J.$$

\end{prop}

\begin{proof}	
	Fixing a vertex $e_J\in \Delta_{k,n}$, let us first compute the coefficient of $e^I$ in Equation \eqref{eq: inversion cubical blades} whenever $I\not=J$.  In this case we find
	\begin{eqnarray*}
		\left(\sum_{L\in C_J}(-1)^{1+(L\cdot L)}\mathcal{R}(e^{J_L})\right)_{\lbrack e^I\rbrack} & = & -\frac{1}{n}\sum_{L\in C_J}(-1)^{1+(L\cdot L) } d(e_{J_L},e_I)\\
		& = & -\frac{1}{n}\sum_{L\in C_J}(-1)^{1+(L\cdot L) } \left(d(e_{J},e_I)-(L\cdot L)\right)\\
		& = & -\frac{1}{n}\sum_{L\in C_J}(-1)^{1+(L\cdot L) } (L\cdot L)\\
		& = & 0.
	\end{eqnarray*}
	Consequently only the coefficient of $e^J$ is (possibly) nonzero; let us now compute it.  The (trivial) path from $e_J$ to itself in steps parallel to roots $e_i-e_{i+1}$ has length zero; all others in the sum contributing to the coefficient of $e^J$ are shortenings of the long path (of length $n$) between $e_J$ and itself and we find that their lengths are of the form $n - L\cdot L$.  Consequently the alternating sum is now $n$ rather than $0$.  We obtain
	\begin{eqnarray*}
		\left(\sum_{L\in C_J}(-1)^{1+(L\cdot L)}\mathcal{R}(e^J)\right)_{\lbrack e^J\rbrack} & = & \frac{1}{n}\sum_{L\in C_J}(-1)^{1+(L\cdot L) } d(e_J,e_{J_L})\\
		& = & \frac{1}{n}\left(n +\sum_{L\in C_J}(-1)^{1+(L\cdot L) } \left(n-(L\cdot L)\right)\right)\\
		& = & 1,
	\end{eqnarray*}
	where in the second line we have added and subtracted $n$.  This concludes the proof of the first claim; the second claim is similar.
\end{proof}

\begin{example}
	In $\mathbb{R}^{\binom{4}{2}}$ we have 
	\begin{eqnarray*}
		\left(\mathcal{R}(e^{24}) - \mathcal{R}(e^{14})-\mathcal{R}(e^{23})+\mathcal{R}(e^{13})\right)_{\lbrack e^{12}\rbrack} & = &  \frac{1}{4}(d(e_{24},e_{12}) - d(e_{14},e_{12}) - d(e_{23},e_{12}) + d(e_{13},e_{12}))\\
		& = & \frac{1}{4}(1-2-2+3)\\
		& = & 0,
	\end{eqnarray*}
	while 
	\begin{eqnarray*}
		\left(\mathcal{R}(e^{24}) - \mathcal{R}(e^{14})-\mathcal{R}(e^{23})+\mathcal{R}(e^{13})\right)_{\lbrack e^{24}\rbrack} & = & \frac{1}{4}(d(e_{24},e_{24}) - d(e_{14},e_{24}) - d(e_{23},e_{24}) + d(e_{13},e_{24}))\\
		& = & \frac{1}{4}(0-3-3+2)\\
		& = & \frac{1}{4}\left(-4 +(4-1)-(4-1)-(4-1)+(4-2) \right)\\
		& = & -1,
	\end{eqnarray*}
	as expected.
\end{example}

\begin{cor}\label{cor:bases vertex space}
	Both sets
	$$\left\{\mathcal{R}(e^J): e_J\in\Delta_{k,n}\right\} \text{ and } \left\{\mathcal{L}(e^J): e_J\in\Delta_{k,n}\right\}$$
	are bases for $\mathbb{R}^{\binom{n}{k}}$.  
	
\end{cor}

\begin{prop}\label{prop: positive tropical Plucker relations}
	For each vertex $e_J\in\Delta_{k,n}$, the element 
	$$\mathfrak{h}_{J} = \sum_{e_I\in\Delta_{k,n}}\rho_J(e_I)e^I$$
	satisfies the following relations: given any $a<b<c<d$ in cyclic order and any $J\in\binom{\lbrack n\rbrack}{k-2}$ such that $e_J+e_{ac},e_J+e_{bd},e_J+e_{ab},e_J+e_{cd},e_J+e_{ad},e_J+e_{cd}$ are all vertices of $\Delta_{k,n}$.  Then
	$$\rho_I(e_J+e_{ac}) + \rho_I(e_J+e_{bd}) = \max\left\{\rho_I(e_J+e_{ab}) + \rho_I(e_J+e_{cd}),\rho_I(e_J+e_{ad}) + \rho_I(e_J+e_{bc})  \right\}.$$
\end{prop}

\begin{proof}
	One could compute this directly, but the geometric argument provides more insight.
	
	The equality above is a direct translation of the statement that if the $\rho_J$ bends nontrivially across an octahedral face in $\Delta_{k,n}$, say $$\partial_L(\Delta_{k,n})=\left\{x\in\Delta_{k,n}: x_{abcd}=2, \text{ and } x_j=1\text{ for all }j\in L \right\}$$ then it so over either of the two affine hyperplanes, $x_{ab}=1=x_{cd}$ or $x_{ad}=1=x_{bc}$, but not both, where we have the cyclic order $a<b<c<d$.  In particular, $h_J$ does not bend across the plane $x_{ac} = 1 = x_{bd}$.  This is exactly what we proved in Section 3 of \cite{Early19WeakSeparationMatroidSubdivision} by giving explicit equations for the internal facet itself.
\end{proof}

\subsection{Planar basis}


We now come to the main construction from \cite{Early2019PlanarBasis}, obtained by dualizing the elements $\mathfrak{h}_{J}$, of an important basis of the space of kinematic invariants: the planar basis.
\begin{defn}
	For any vertex $e_J\in\Delta_{k,n}$, define the planar (basis) element
	\begin{eqnarray}\label{eq:deltaJ}
	\eta_J(s) = -\frac{1}{n}\sum_{e_I \in\Delta_{k,n}}s_I \rho_J(e_I).
	\end{eqnarray}
\end{defn}
One can easily see that the set of these elements $\eta_J$ are invariant under cyclic permutation by the cycle $(12\cdots n)$.

We have the following straightforward property for the functions $\rho_{J_i}$

\begin{prop}\label{prop:frozen constant slope}
	For any frozen vertex $e_{J_i} \in \Delta_{k,n}$, the function $\rho_{J_i}$ has constant slope over $\Delta_{k,n}$.  We have that $\eta_{J_i} \equiv 0$ on the kinematic space $\mathcal{K}_{k,n}$.
\end{prop}

\begin{proof}
	When $e_{J_i}$ is a frozen vertex of $\Delta_{k,n}$, then according to Theorem 17 of \cite{Early19WeakSeparationMatroidSubdivision}, the blade $\beta_{J_i}$ induces the trivial subdivision of $\Delta_{k,n}$; this means that its lift $\mathfrak{h}_{e_{J_i}} \in \mathbb{R}^{\binom{n}{k}}$ is linear over the vertices of $\Delta_{k,n}$.
	
	But the subspace of elements of $\mathbb{R}^{\binom{n}{k}}$ that are linear over $\Delta_{k,n}$ also has basis the $n$ elements
	$$\left\{\sum_{J\ni j}e^J:j=1,\ldots, n\right\}.$$
	Dualizing this one obtains exactly the elements 
	$$\left\{\sum_{J\ni j}s_J:j=1,\ldots, n\right\},$$
	and imposing \textit{momentum conservation}
	$$\sum_{J \ni j}s_J = 0$$
	for each $j=1,\ldots, n$ characterizes the kinematic space $\mathcal{K}_{k,n}$ and it follows that each $\eta_{J_i}\equiv 0$.
\end{proof}

Our main result of this section is the following.

\begin{thm}\label{thm: planar basis}
	The set 
	$$\left\{\eta_J: e_J\in\Delta_{k,n}\text{ is nonfrozen} \right\}$$
	is a basis of linear kinematic functions on $\mathcal{K}_{k,n}$.
\end{thm}
\begin{proof}
	Apply Corollary \ref{cor:bases vertex space}, replacing $\mathcal{R}(e^J) $ with $\eta_J$ and $\mathcal{L}(e^J)$ with $s_J$.
\end{proof}
 Note also that by Corollary \ref{prop:frozen constant slope} all other planar basis elements (i.e. those that are frozen) are identically zero.

\begin{example}
	The two cases for Mandelstam invariants on $\mathcal{K}(3,6)$ in terms of planar basis functions are as follows.
	\begin{eqnarray*}
		s_{i,j-1,j} & = & -(\eta_{i,j-1,j}+\eta_{i-1,j-2,j}) + (\eta_{i-1,j-1,j} + \eta_{i,j-2,j})\\
	\end{eqnarray*}
and when $i,j,k$ are not cyclically adjacent,
\begin{eqnarray*}
		s_{i,j,k} & = & \eta _{i-1,j-1,k-1}-\eta _{i-1,j-1,k}-\eta _{i-1,j,k-1}+\eta _{i-1,j,k}-\eta _{i,j-1,k-1}+\eta _{i,j-1,k}+\eta _{i,j,k-1}-\eta _{i,j,k}.
\end{eqnarray*}
\end{example}
Clearly, this specializes to formulas used previously on the kinematic space $\mathcal{K}(2,n)$, for instance \cite{ABHY}.

\section{Hypersimplicial Blade complex}\label{sec: hypersimplicial blade hypersimplex}
In this section, we study combinatorial properties of a graded vector space which consists of formal linear combinations of translated blades $\beta_J(\alpha) = ((\alpha_1,\alpha_2,\ldots, \alpha_n))_{e_J}$, together with a hypersimplex $\Delta_{k,n}$ and a set of $n$ boundary operators $\partial_j$; blades will satisfy certain relations prescribed by their interactions with the faces of the hypersimplex.  

We take $\alpha$ to be the standard cyclic order $(1,2,\ldots n)$, and write just $\beta_J$.

Let us turn our attention to homological properties of the symbols $\beta_J$, as well as their images $\beta^{(L)}_J$ under a certain set of linear boundary operators $\partial_j$, where $L$ is any subset of $\{1,\ldots, n\}$ , and $J\in\binom{\lbrack n\rbrack\setminus L}{k-\vert L\vert}$ is arbitrary as well.  Here, when $L=\emptyset$ is the empty set (and $\alpha = (12\cdots n)$) we write simply $\beta_J$.  The intuition, that $\beta^{(L)}_J$ is the curvature induced by the blade $((1,2,\ldots, n))_{e_J}$ on the face $\partial_L(\Delta_{k,n})$, will inform the linear relations.

We set $\beta^{(L)}_J=0$ when no subdivision is induced by $\beta_J$ on the face $\partial_L(\Delta_{k,n})$.  This is the case exactly when $J$ is frozen with respect to the gapped cyclic order on $\{1,\ldots, n\}\setminus L$ inherited from $\alpha$.  

There is a natural action induced by restriction: define linear operators $\partial_j$ on the linear span of the symbols $\beta^{(L)}_J$ as follows.  

\begin{itemize}
	\item If $j\in L$, then we set $\partial_j(\beta^{(L)}_J) = 0$.  
	\item If $j\not\in L$, set 
	$$\partial_j(\beta^{(L)}_J) = \beta^{(L\cup \{j\})}_{J\setminus\{\ell\}},$$
	where $\ell=j$ if $j\in J$, and otherwise $\ell$ is the cyclically next element of $\{1,\ldots, n\}$ that is in $J$.
\end{itemize}

Define 
$$\partial = \partial_1+\cdots + \partial_n.$$

Then we have the operator-theoretic identities for powers,
$$\partial_j^2=0,\ \text{ and } \frac{1}{d!}\partial^{d} = \sum_{L\in\binom{\lbrack n\rbrack}{d}}\partial_{L},$$
where we have defined $\partial_L = \partial_{\ell_1}\cdots \partial_{\ell_d}$ with $\ell_1<\cdots<\ell_d$, when $L=\{\ell_1,\ldots, \ell_d\}$.

We take the ``cyclically next element'' in order to match the notation used to encode the subdivision induced on the boundary, as in \cite{Early19WeakSeparationMatroidSubdivision}; in this way our construction is not ad hoc; it is strictly determined geometrically.

This will be more clear with an example.

\begin{example}
	Let $J = \{1,4,5,6\}$, with $n=8$.  Then $\partial_{1}(\beta_{1456}) = 0$, since $456$ is a single cyclic interval in $\{2,\ldots, 6\}$.  Here the intuition is that we have $\beta^{(1)}_{456}=0$ because  it corresponds to the curvature of a continuous, (piecewise-)linear function over the hypersimplex $\partial_1(\Delta_{4,8})\simeq \Delta_{3,7}$, which has in fact zero curvature on the interior and consequently induces the trivial subdivision.  In particular, it is a linear function, not only piecewise-linear.  But $\partial_2(\beta_{1456}) = \beta^{(2)}_{156}$ is not zero, since $156$ is not a cyclic interval in $\{1,3,4,5,6,7,8\}$. 
\end{example}

Recall from \cite{Early19WeakSeparationMatroidSubdivision} that any hypersimplicial blade coincides locally with translated copy of $((\alpha_1,\ldots, \alpha_n))$ for some cyclic order $(\alpha_1,\ldots, \alpha_n)$ of $\lbrack n\rbrack$.
\begin{prop}[\cite{Early19WeakSeparationMatroidSubdivision}]
	Given any hypersimplicial blade $(((S_1)_{s_1},\ldots, (S_\ell)_{s_\ell}))$, then there exist a cyclic order $\alpha$ of $\lbrack n\rbrack$ and a vertex $e_J\in\Delta_{k,n}$ such that we have the local coincidence
	$$((\alpha_1,\ldots, \alpha_n))_{e_J}\cap \Delta_{k,n} = (((S_1)_{s_1},\ldots, (S_\ell)_{s_\ell}))\cap \Delta_{k,n}.$$
\end{prop}

Motivated in part by the fact that the generalized biadjoint scalar $m^{(k)}(\alpha,\alpha)$ allows the cyclic order to vary, but also for sake of generality, Definition \ref{defn: blade complexes} constructs a larger space for any cyclic order $\alpha$; however, our main results in this paper require the single cyclic order $\alpha = (1,2,\ldots, n)$.  Therefore in Definition \ref{defn: blade complexes} all blade arrangements involve only translations of the blade $((1,2,\ldots, n))$.
\begin{defn}\label{defn: blade complexes}
	Denote 
	$$\mathfrak{B}^\bullet_{k,n}=\bigoplus_{m=0}^{n-(k-2)}\left(\bigoplus_{L\in\binom{\lbrack n\rbrack}{m}}\mathfrak{B}^{(L)}_{k,n}\right),$$
	where $\mathfrak{B}^{(L)}_{k,n}$ is the set of formal linear combinations 
$$\mathfrak{B}^{(L)}_{k,n} = \text{span}\left\{\beta^{(L)}_J: J\in\binom{\lbrack n\rbrack \setminus L}{k-\vert L\vert}\text{ is nonfrozen}\right\}.$$
	Further denote
	$$\mathfrak{B}^{m}_{k,n} = \bigoplus_{L\in\binom{\lbrack n\rbrack}{m}}\mathfrak{B}^{(L)}_{k,n}$$
	for integers $m = 0,1,\ldots, n-(k-2)$.

\end{defn}
Denote by
$$\text{supp}\left( \sum_{\{\{i,j\}\subset \left(\lbrack n\rbrack\setminus L\right)\}} \pi^{(L)}_{ij}\beta^{(L)}_{ij}\right)$$
the set of $\beta^{(L)}_{ij}$ such that $\pi^{(L)}_{ij} \not=0$ in the linear combination.

\begin{rem}
	Also of interest is the more subtle ``master'' space which is obtained from the spaces $\mathfrak{B}^\bullet_{k,n}(\alpha)$ as $\alpha$ varies over all cyclic orders.  This is nontrivial: geometrically this is because a positroidal subdivision can be $\alpha$-planar with respect to several different $\alpha$'s!

	Here the distinguished elements, which are in bijection with multi-split matroidal subdivisions, were enumerated in Corollary \ref{cor:hypersimplicial blades count}: they are counted by the Eulerian numbers $A_{k-1,n-1}$.
\end{rem}

However, note that even this master space which characterizes all possible generalized Feynman diagrams appearing in any $m^{(k)}(\alpha,\alpha)$ as $\alpha$ varies over all cyclic orders of $\lbrack n\rbrack$, is still not the whole tropical Grassmannian.  
\begin{example}\label{example: nonplanar subdivision}
	There is no single cyclic order $\alpha$ on $\lbrack n\rbrack$ such that the matroidal subdivision of $\Delta_{3,6}$ induced by the hypersimplicial blades
	$$\left\{((123_2 456_1)),((126_1 345_2)), ((156_2 234_1)),((135_1 246_2))\right\},$$
	is induced by a (weighted) matroidal blade arrangement on some four vertices of $\Delta_{3,6}$, but from \cite{SpeyerStumfelsTropGrass} this does induce a maximal cone in $\text{Trop }G(3,6)$; it induces the cone of type EEEE in their notation.
\end{example}

The grading on the space $\mathfrak{B}^{m}_{k,n}$ (i.e., with $\alpha = (1,2,\ldots, n)$) is understood to correspond to the ambient codimension $m=\vert L\vert$ for the curvature $\beta^{(L)}_J\cap \partial_L(\Delta_{k,n})$ of the faces of $\Delta_{k,n}$, and the linear operators $\partial_j$ are to be understood roughly to correspond to restriction of the curvature to the face $x_j=1$ of $\Delta_{k,n}$.

It is now immediate that 
$$\partial_j:\mathfrak{B}^{(L)}_{k,n} \rightarrow \mathfrak{B}^{(L\cup \{j\})}_{k,n},$$
where the image is trivial when $j\in L$.  Directly from the definition we see that the top component $\mathfrak{B}_{k,n}$ of $\mathfrak{B}^\bullet_{k,n}$ satisfies $\dim\left(\mathfrak{B}_{k,n}\right) = \binom{n}{k} - n$, with basis $\left\{\beta_J: e_J\in\Delta_{k,n} \text{ is nonfrozen} \right\}$.

\begin{prop}\label{prop:boundary trivial kernel}
	The boundary map $\partial = \partial_1+\cdots +\partial_n$ is injective: we have
	$$\bigcap_{j=1}^n \ker \partial_j = \{0\}.$$
\end{prop}

\begin{proof}
	The proof amounts to correctly interpreting definitions.
		
	By construction, for each $j=1,\ldots, n$, the space $ \partial_j(\mathfrak{B}_{k,n})$ is freely spanned by the set of $\binom{n-1}{k-1}-(n-1)$ elements
	$$\left\{\beta^{(j)}_J: J\in \binom{\lbrack n\rbrack \setminus \{j\}}{k-1}\text{ is nonfrozen with respect to the cyclic order }(1,2,\ldots, \hat{j},\ldots, n) \right\}.$$
	It is easy to check from the definition of the boundary operator that $\partial_j(\beta_J) = 0$ if and only if $J$ is the union of a $k-1$ element cyclic interval together with a singlet, of the following form:
	$$J = \{i,i+1,\ldots, i+(k-2)\} \cup \{a\}$$
	for some $a\in \{j,j+1,\ldots, i-1\}$ with empty intersection $J\cap \{j,j+1,\ldots, i-1\} = \emptyset$.  Thus $\dim\ker(\partial_j) = \binom{n}{2}-n$ after excluding subsets $J$ that are already frozen, since then $\beta_J=0$.

	The intersection is then the span of the $\beta_J$ where $J$ is frozen, hence 
	$$\bigcap_{j=1}^n \ker\partial_j = \{0\}.$$
	
\end{proof}

\subsection{Localized Spanning set for $\mathcal{B}_{k,n}$}
In this section, in the process of defining elements $\mathcal{L}_J \in \mathfrak{B}_{k,n}$, we emphasize that we are fixing once and for all a cyclic order $\alpha = (12\cdots n)$ and the ambient polytope, the hypersimplex $\Delta_{k,n}$.

So define $\mathcal{L}_J\in \mathfrak{B}_{k,n}$ by 
$$\mathcal{L}_J = \sum_{M\in C_J}(-1)^{1+(M\cdot M) }\beta_{J_M},$$
where $C_J$ is the cube
$$C_J=\left\{J_M=\{j_1-m_1,\ldots, j_t-m_t\}: M=(m_1,\ldots, m_t) \in \{0,1\}^t\right\},$$
with $\{j_1,\ldots, j_t\}$ being the cyclic initial points of the cyclic intervals of $J$.

\begin{prop}\label{prop: inversion blades}
	We have 
	$$\sum_{I}\rho_J(e_J)\mathcal{L}_I = \beta_J.$$
\end{prop}
\begin{proof}
	This is a direct consequence of the formula in Proposition \ref{cor:inversionFormula}.
\end{proof}

The definition for higher codimension faces of $\Delta_{k,n}$ is precisely analogous.  Namely, for $J \in \binom{\lbrack n\rbrack \setminus L}{k-\vert L\vert}$, put 
$$\mathcal{L}^{(L)}_J= \sum_{M\in C_{J}}(-1)^{1+(M\cdot M) }\beta_{J_M}.$$
The difference here is that the calculation of the sets $J_M$ is slightly more technical as the cyclic order on $\{1,\ldots, n\}\setminus L$ is now ``gapped.''

\begin{prop}\label{prop: Greens function property blade}
	Given $J\in \binom{\lbrack n\rbrack}{k}$ and a (proper) subset $L \subset \lbrack n\rbrack$, then 
	$$\partial_L\left(\mathcal{L}_J\right) = \begin{cases}
	\mathcal{L}^{(L)}_J& L\subseteq J\\
	0 & L\not\subset J
	\end{cases}$$
	In particular, if now $j\in \lbrack n\rbrack$, then 
	$$\partial_j\left(\mathcal{L}_J\right) = \begin{cases}
	\mathcal{L}^{\{j\}}_J& j\in J\\
	0 & j\not\in J
	\end{cases}$$
\end{prop}
\begin{proof}[Idea of proof]
	This requires some straightforward bookkeeping, which we omit, with terms that cancel in pairs in the expansion of $\partial_j (\mathcal{L}_J)$.  For the $\partial_L\left(\mathcal{L}_J\right)$ identity simply iterate the first argument, since $\partial_L = \partial_{\ell_1}\cdots \partial_{\ell_k}$.
\end{proof}

\subsection{Hypersimplicial blade complexes and the positive tropical Plucker relations}\label{sec: hypersimplex blade complex}
As usual, we fix a cyclic order $\alpha = (\alpha_1,\alpha_2,\ldots, \alpha_n)$.  In particular, when $\alpha = (1,2,\ldots, n)$ we shall revert to the standard terminology: $\alpha$-planar subdivisions are \textit{alcoved} and $\alpha$-planar matroid subdivisions are \textit{positroidal}.

In \cite{Early19WeakSeparationMatroidSubdivision} the notion of a blade arrangement (on the vertices of a hypersimplex $\Delta_{k,n}$) was define: it is a set of blades $\{\beta_{J_1},\ldots, \beta_{J_m}\}$ arranged on the vertices $e_{J_1},\ldots, e_{J_m}$ of a hypersimplex $\Delta_{k,n}$; it was shown that their union subdivides $\Delta_{k,n}$ into a number of alcoved polytopes.

We bring forward the following result from the introductory discussion.
\begin{thm}[\cite{Early19WeakSeparationMatroidSubdivision}]
	The subdivision of $\Delta_{k,n}$ that is induced by a blade arrangement 
	$$\{\beta_{J_1},\ldots, \beta_{J_m}\}$$ 
	is matroidal (in particular positroidal) if and only if the collection $\{J_1,\ldots, J_m\}$ is weakly separated.
\end{thm}
However, this procedure achieves a relatively small subclass of positroidal subdivisions.  In this section, we show how to achieve more general positroidal subdivisions, by attaching a weight to each blade in an arrangement; practically speaking, we take formal linear combinations  of $\beta_J$'s.  In fact, from Lemma \ref{lem:posDressToBlades}, together with results from \cite{arkani2020positive, SpeyerWilliams2020} it follows that any regular positroidal subdivision can be achieved in this way, via a lifting function.

Remark that the computations in \cite{Early2019PlanarBasis} of the numbers of maximal cones in $m^{(k)}(\alpha,\alpha)$ for various values of $(k,n)$ were obtained by mapping a positive tropical Plucker vector to a linear combination of blades $\beta_J$ in anticipation of Lemma \ref{lem:posDressToBlades}.  Here $\alpha = (12\cdots n)$ is the standard cyclic order on $\{1,\ldots, n\}$.

Directly translating our construction of the boundary operators $\partial_i$ gives the coefficients $\pi^{(L)}_{ij}$ in
$$\partial_{L}\left(\sum_{J\in\binom{\lbrack n\rbrack}{k}}c_J\beta_J(x)\right) = \sum_{\{\{i,j\}\subset \left(\lbrack n\rbrack\setminus L\right)\}} \pi^{(L)}_{ij}\beta^{(L)}_{ij},$$
that is
$$\pi^{(L)}_{ij} = \sum_{\left\{I\in \binom{\lbrack n\rbrack}{k}:\ \partial_L\left(\beta_I\right) = \beta_{ij}^{(L)}\right\}}c_I$$ 
where our coefficients are real numbers $c_I\in\mathbb{R}$.

Recall the notation
$$\text{supp}\left( \sum_{\{\{i,j\}\subset \left(\lbrack n\rbrack\setminus L\right)\}} \pi^{(L)}_{ij}\beta^{(L)}_{ij}\right)$$
for the set of $\beta^{(L)}_{ij}$ such that $\pi^{(L)}_{ij} \not=0$ in the linear combination.

This induces for each such $L$ a subdivision of the corresponding second hypersimplicial face $\mathcal{F}_{L}\simeq \Delta_{2,n-(k-2)}$; we are interested in the case when this subdivision is matroidal (in which case, when we have fixed the standard cyclic order $(12\cdots n)$, it is positroidral).  

\begin{rem}
	Specializing Theorem 31 of \cite{Early19WeakSeparationMatroidSubdivision} to the usual case $k=2$, we have that the subdivision of $\partial_L(\Delta_{k,n})\simeq\Delta_{2,n-(k-2)}$ induced by the blade arrangement $\{\beta^{(L)}_{i_1j_1},\ldots, \beta^{(L)}_{i_mj_m}\}$ is positroidal if and only if the collection of pairs $\{\{i_1,j_1\},\ldots, \{i_m,j_m\}\}$ is weakly separated with respect to the cyclic order inherited on $\lbrack n\rbrack \setminus L$ from the cyclic order $(1,2,\ldots, n)$.
\end{rem}

Let us use the notation 
$$\beta(\mathbf{c}) = \sum_{\left\{J\in\binom{\lbrack n\rbrack}{k}\right\}\text{ is not frozen}}c_J\beta_J(x) \in \mathfrak{B}_{k,n}$$
for a given coefficient tensor $(\mathbf{c}_J)\in\mathbb{R}^{\binom{n}{k}}$ (however, usually we will take the $\mathbf{c}_J$ to be rational numbers or even integers.  It will be clear from context what we mean).  Here we remind that $\beta_J=0$ whenever $J$ is frozen.

Denote by
$$\text{supp}\left( \sum_{\{\{i,j\}\subset \left(\lbrack n\rbrack\setminus L\right)\}} \pi^{(L)}_{ij}\beta^{(L)}_{ij}\right)$$
the set of $\beta^{(L)}_{ij}$ such that $\pi^{(L)}_{ij} \not=0$ in the linear combination.

	We fix the cyclic order $\alpha = (12\cdots n)$ on the set $\lbrack n\rbrack=\{1,\ldots, n\}$, as usual.

\begin{defn}\label{defn:Xkn YKn ZKn}

	Denote by 
	$$\mathcal{X}_{k,n} = \left\{\beta(\mathbf{c}) \in \mathfrak{B}_{k,n}: \text{supp }\partial_L(\beta_{\mathbf{c}})\text{ is a matroidal blade arrangement on } \partial_L(\Delta_{k,n}),L\in\binom{\lbrack n\rbrack}{k-2}\right\}$$
	the arrangement of (real) subspaces in $\mathfrak{B}_{k,n}$ consisting of linear combinations of blades $\beta_J$ which have $\alpha$-planar support, and by 
	$$\mathcal{Y}_{k,n}= \left\{\beta(\mathbf{c})  \in \mathfrak{B}_{k,n}: \text{coeff}\left(\partial_L(\beta(\mathbf{c})),\beta^L_{ij}\right)\ge0,L\in\binom{\lbrack n\rbrack}{k-2},\ \{i,j\}\in\binom{\lbrack n\rbrack\setminus L}{2}\right\},$$
	the convex cone in $\mathfrak{B}_{k,n}$ with nonnegative curvature on every second hypersimplicial face $\partial_L(\Delta_{k,n})$.  

	Let $\mathcal{Z}_{k,n}$ denote their intersection:
$$\mathcal{Z}_{k,n}= \mathcal{X}_{k,n} \cap \mathcal{Y}_{k,n}.$$

\end{defn}

The main result in this section, Lemma \ref{lem:posDressToBlades}, is that the defining relations for $\mathcal{Z}_{k,n}$ can be rewritten in terms of exactly the positive tropical Plucker relations!  This immediately implies a direct connection to the tropical Grassmannian, via the very recent work \cite{arkani2020positive,SpeyerWilliams2020}.  We return to this in Lemma \ref{lem:posDressToBlades}.

In connection with Lemma \ref{lem:posDressToBlades}, see also the computations in \cite{Early2019PlanarBasis} of blade arrangements and the planar basis, and works of Borges-Cachazo \cite{BC2019} and Cachazo-Guevara-Umbert-Zhang \cite{CGUZ2019} on  generalized Feynman diagrams, where the finest metric tree collections and metric tree arrays, respectively, were calculated for the second hypersimplicial faces of $\Delta_{k,n}$ and used to derive the full positive tropical Grassmannians through $\text{Trop}_+G(4,9)$, and consequently the generalized biadjoint scalars $m^{(k)}(\alpha,\alpha)$ for $$(k,n)\in\{(3,6),(3,7),(3,8),(3,9),(4,8),(4,9)\}.$$  These structures were also studied in \cite{HeRenZhang2020} where the set of possible poles was given additionally for the $n=10$ particle $m^{(3)}(\alpha,\alpha)$, using codimension one boundary configurations for the moduli space $X(3,10)$ of $10$ points in $\mathbb{CP}^{2}$ moduli $PGL(3)$.

To prove Lemma \ref{lem:posDressToBlades}, we will first need the expression for the boundary of an arbitrary linear combination of the spanning set 
$\{\mathcal{L}_J: e_J\in \Delta_{k,n}\}$ for the space $\mathfrak{B}_{k,n}$.  
\begin{cor}\label{cor:boundary local spanning set}
	Given any element $\beta(\mathbf{c})= \sum_{J}c_J \mathcal{L}_J\in\mathfrak{B}_{k,n}$, then corresponding to any second hypersimplicial face $\partial_L(\Delta_{k,n}) \simeq \Delta_{2,n-(k-2)}$ we have 
	$$\partial_L(\beta(\mathbf{c})) = \sum_{\{i,j\}\in\binom{\lbrack n\rbrack \setminus L}{2}}\omega^{(L)}_{ij} \beta^{(L)}_{ij},$$
	where the sum is over nonfrozen pairs of distinct elements $i,j\in\{1,\ldots, n\}\setminus L$, with	$\omega^{(L)}_{ij}$ given by 
	$$\omega^{(L)}_{ij} = -(c_{Lij} - c_{L,i,j+1} - c_{L,i+1,j} + c_{L,i+1,j+1}).$$
\end{cor}

\begin{proof}
	Apply Proposition \ref{prop: Greens function property blade} to compute $\partial_L(\beta(\mathbf{c}))$ and then expand the result in terms of $\beta^{(L)}_{ij}$'s.
\end{proof}

Here the blades $\beta_J$ by construction correspond to the discrete curvature of piecewise linear surfaces over the hypersimplex $\Delta_{k,n}$.  Lemma \ref{lem:posDressToBlades} says that the nonnegative tropical Grassmannian coincides with all linear combinations of blades which pass via $\partial_L$ to a nonnegative linear combination of basis elements $\beta^{(L)}_J$ on the second hypersimplicial faces $\partial_L(\Delta_{k,n})$. 

Example \ref{example: W 37} illustrates how the cancellation works on these faces; it is a useful exercise to fill in the details.

\begin{example}\label{example: W 37}
	Let us consider an example to give some insight into how to determine which positroidal subdivisions are induced on the second hypersimplicial faces $\partial_j(\Delta_{3,7})$.  Consider the element $-\beta_{247} + \beta_{124} + \beta_{257} + \beta_{347} \in \mathfrak{B}_{3,7}$.  Taking the boundary we find 
	\begin{eqnarray*}
		\partial(-\beta_{247} + \beta_{124} + \beta_{257} + \beta_{347}) & = & \beta^{(1)}_{24}+\beta^{(1)}_{57} + \beta^{(2)}_{14} + \beta^{(2)}_{57} + \beta^{(3)}_{37} + \beta^{(4)}_{37} + \beta^{(5)}_{24} + \beta^{(6)}_{25} + \beta^{(7)}_{25},
	\end{eqnarray*}
	Here all coefficients are positive, and the support on each facet $x_j=1$ of $\Delta_{3,7}$ is a matroidal blade arrangement, so we confirm that $(-\beta_{247} + \beta_{124} + \beta_{257} + \beta_{347})\in\mathcal{Z}_{3,7}$.
	
	Then, the element $(-\beta_{247} + \beta_{124} + \beta_{257} + \beta_{347})$ induces on the face $\partial_1(\Delta_{3,7}) \simeq\Delta_{2,6}$, for instance, a superposition of two (compatible) 2-splits, giving a positroidal subdivision with maximal cells separated by the blades 
	$$((34_1 5672_1))^{(1)}\text{ and } ((2345_1 67_1))^{(1)},$$
	corresponding to respectively $\beta_{24}^{(1)}$ and $\beta^{(1)}_{57}$.  Similarly one can find the subdivisions induced on the other facets, and construct the whole tree arrangement, as in \cite{Hermann How to draw}, see also \cite{BC2019,CGUZ2019}.
\end{example}

\begin{lem}\label{lem:posDressToBlades}
	Let $\pi(\mathbf{c})= \sum_{J}c_J e^J\in\mathbb{R}^{\binom{n}{k}}$.  Then, $\pi(\mathbf{c})$ satisfies the positive tropical Plucker relation
	$$c_{Lj_1j_3} + c_{Lj_2j_4} = \max\left\{c_{Lj_1j_2} + c_{Lj_3j_4}, c_{Lj_1j_4}+c_{Lj_2j_3}\right\}$$
	for each cyclic order $j_1<j_2<j_3<j_4$, where the elements $e_{Lj_aj_b}$ are the six vertices of any given octahedral face of $\Delta_{k,n}$, if and only if the weighted blade arrangement $\beta(\mathbf{c})=\sum_J c_J \mathcal{L}_J \in \mathfrak{B}_{k,n}$ is in $\mathcal{Z}_{k,n}$.

	Moreover, modding out $\mathbb{R}^{\binom{n}{k}}$ by lineality gives a strict bijection between the positive tropical Grassmannian $\text{Trop}_+ G(k,n)$ and $\mathcal{Z}_{k,n}$.
\end{lem}
\begin{proof}
	Suppose that $\pi(\mathbf{c})\in\mathbb{R}^{\binom{n}{k}}$ satisfies the nonnegative tropical Plucker relations.  Then for any cyclic order $\ell<\ell+1 < m < m+1$, we have in particular
	$$\left(c_{L\ell m} - c_{L,\ell+1,m\ } -c_{L,\ell,m+1} + c_{L,\ell+1,m+1}\right)\le 0.$$ 	
	By Corollary \ref{cor:boundary local spanning set} we recognize this as $-\omega^{(L)}_{\ell m}$ hence $\sum_J c_J \mathcal{L}_J\in\mathcal{Y}_{k,n}$, since
	$$\partial_{L}\left(\beta(\mathbf{c})\right) = \sum_{}\omega^{(L)}_{ij}\beta^{(L)}_{ij}$$
	with $\omega^{(L)}_{\ell m} \ge 0$.  It now remains to prove that the support of each $\partial_L(\beta(\mathbf{c}))$ is a matroidal blade arrangement.
	
	It is easy to see that the latter requirement, together with $\omega^{(L)}_{\ell m} \ge 0$, have the more compact expression
	\begin{eqnarray}\label{eqn:pairsNotWS}
		\min\left\{\omega^{(L)}_{ij},\sum_{\{\{i,j\},\{a,b\}\}\text{ not w.s.}} \omega^{(L)}_{ab}\right\}=0
	\end{eqnarray}
	for each given pair of distinct (nonfrozen) elements $\{i,j\}$, and where the sum is over all $\{a,b\}\subset \{1,\ldots, n\}\setminus L$ such that $\{\{i,j\},\{a,b\}\}$ is not weakly separated.  
	
	From the identity 
	$$\sum_{\{\{i,j\},\{a,b\}\}\text{ not w.s.}} \omega^{(L)}_{ab} = -(c_{Li,i+1} - c_{L,i,j+1} - c_{L,i+1,j}+c_{L,j,j+1}),$$
	it follows that 
	\begin{eqnarray*}
	& &  \min\left\{-(c_{Lij} - c_{Li,j+1} - c_{Li+1,j} + c_{L,i+1,j+1}),-(c_{Li,i+1} - c_{L,i,j+1} - c_{L,i+1,j}+c_{L,j,j+1})\right\}\\
	& = & \min\left\{-(c_{Lij} + c_{L,i+1,j+1}),-(c_{Li,i+1} +c_{L,j,j+1})\right\} + (c_{Li,j+1} + c_{Li+1,j}).
\end{eqnarray*}
Multiplying by $(-1)$ and applying the nonnegative tropical Plucker relation gives
$$\max\left\{(c_{Lij} + c_{L,i+1,j+1}),(c_{Li,i+1} +c_{L,j,j+1})\right\} - (c_{Li,j+1} + c_{Li+1,j}) = 0.$$ 
Consequently Equation \eqref{eqn:pairsNotWS} holds, hence $\sum_J c_J \mathcal{L}_J\in\mathcal{Z}_{k,n}$.

	Conversely, if $\pi(\mathbf{c})\in\mathbb{R}^{\binom{n}{k}}$ such that $\beta(\mathbf{c})\in\mathcal{Z}_{k,n}$, then all of the coefficients in the expansion of $\partial_L(\beta(\mathbf{c}))$ satisfy Equation \eqref{eqn:pairsNotWS} modulo relabeling.

	We claim that the nonnegative tropical Plucker relation 
	$$c_{Lj_1j_3} + c_{Lj_2j_4} = \max\left\{c_{Lj_1j_2} + c_{Lj_3j_4}, c_{Lj_1j_4}+c_{Lj_2j_3}\right\}$$
	holds for each cyclic order $j_1<j_2<j_3<j_4$; therefore the case $(j_1,j_2,j_3,j_4) = (\ell,\ell+1,m,m+1)$ is already implied by Equation \eqref{eqn:pairsNotWS}.
	
	In light of the general relation
	\begin{eqnarray*}
				c_{j_1j_3} - c_{j_1j_4} - c_{j_2j_3} + c_{j_2j_4} & = & \sum_{(a,b)\in\lbrack j_1,j_2-1\rbrack \times \lbrack j_3,j_4-1\rbrack} \omega_{a,b}.
			\end{eqnarray*}
	it remains to check that 
	\begin{eqnarray}
		\min\left\{ \sum_{(a,b)\in(\lbrack j_1,j_2-1\rbrack\setminus L) \times (\lbrack j_3,j_4-1\rbrack\setminus L)} \omega^{(L)}_{a,b}, \sum_{(a,b)\in\lbrack j_1,j_2-1\rbrack \times \lbrack j_3,j_4-1\rbrack} \left(\sum_{\{\{c,d\},\{a,b\}\}\text{ not w.s.}} \omega^{(L)}_{cd}\right)\right\} & =& 0 \label{eqn:big min}.
	\end{eqnarray}
	If $\sum_{(a,b)\in\lbrack j_1,j_2-1\rbrack \times \lbrack j_3,j_4-1\rbrack} \omega_{a,b}=0$ then since all $\omega^{(L)}_{ij}$ are already nonnegative there is nothing to show, so suppose that it is positive.  Then some $\omega^{(L)}_{a,b}>0$, which implies that we must have correspondingly
	$$\sum_{\{\{c,d\},\{a,b\}\}\text{ not w.s.}} \omega^{(L)}_{cd}=0.$$
	Repeat this procedure until either argument of Equation \eqref{eqn:big min} is zero and the whole equation vanishes.  
	
	Finally, from Proposition \ref{prop:frozen constant slope} one can see that the kernel of the homomorphism $\mathbb{R}^{\binom{n}{k}}\rightarrow \mathfrak{B}_{k,n}$ defined by $e^J\mapsto \mathcal{L}_J$ is the so-called lineality space, that is the space of functions over (the vertices of) $\Delta_{k,n}$ that have constant slope, and thus zero curvature.  That is, we obtain an isomorphism
	$$\mathbb{R}^{\binom{n}{k}}\big\slash \text{span}\left\{\sum_{\left\{J\in\binom{\lbrack n\rbrack}{k}:\ a\ni J \right\}}e^{J}: a=1,\ldots, n \right\}\rightarrow \mathfrak{B}_{k,n} ,$$
	that is induced on the quotient by 
	$$e^J\mapsto \mathcal{L}_J.$$
	This completes the proof.
\end{proof}
It was shown independently \cite{arkani2020positive,SpeyerWilliams2020} that the positive Dressian is the set of points in $\mathbb{R}^{\binom{n}{k}}$ that satisfy the positive tropical Plucker relations.

\begin{thm}[\cite{arkani2020positive,SpeyerWilliams2020}]\label{thm:positiveDressianTropGrass}
	The positive tropical Grassmannian is equal to the positive Dressian.
\end{thm}

Our main result is Theorem \ref{thm:tropGrass onto Zkn}, which provides a novel interpretation of the positive tropical Grassmannian, in the complex of matroidal blade arrangements that have (1) positive and (2) weakly separated support on all second hypersimplicial boundaries.

\begin{thm}\label{thm:tropGrass onto Zkn}
	The positive tropical Grassmannian maps surjectively onto $\mathcal{Z}_{k,n}$ with fiber the lineality space.
\end{thm}

\begin{proof}
	Lemma \ref{lem:posDressToBlades} shows that the defining relations for $\mathcal{Z}_{k,n}$ are equivalent to the positive tropical Plucker relations.  This says that $\mathcal{Z}_{k,n}$ is isomorphic to the positive Dressian modulo lineality.  Now apply Theorem \ref{thm:positiveDressianTropGrass}.
\end{proof}

\begin{example}
	We illustrate for the blade parametization the equations on one of the boundaries of an arbitrary linear combination $$\beta(\mathbf{c}) = \sum_{e_{abc}\in\Delta_{3,6}}c_{abc}\beta_{abc}$$ of blades on the (nonfrozen) vertices of $\Delta_{3,6}$.  Here we emphasize that the coefficients $c_{abc}$ need not be all positive in order to have $\beta(\mathbf{c})\in\mathcal{Z}_{3,6}$.  Then 
	$$\partial_6(\beta(\mathbf{c})) = c_{136}\beta^{(6)}_{13} + c_{146}\beta^{(6)} _{14} + \left(c_{124}+c_{246}\right)\beta^{(6)} _{24}+\left(c_{125}+c_{256}\right)\beta^{(6)} _{25} + \left(c_{135}+c_{235}+c_{356}\right)\beta^{(6)} _{35}.$$
	We consider two out of the five equations for the facet $\partial_6(\Delta_{3,6})$.  From the coefficient of $\beta^{(6)}_{14}$ we extract the following condition:
	$$\min\left\{\omega^{(6)}_{14}, \omega^{(6)}_{25}+\omega^{(6)}_{35}\right\}=0,$$
	that is 
	$$\min\left\{c_{146}, c_{125}+c_{256}+c_{135} + c_{235}+c_{356}\right\}=0,$$
	after substituting in the values of the $\omega^{(\ell)}_{ab}$.
	
	Now, from the coefficient $\omega^{(6)}_{35}$ of $\beta^{(6)}_{35}$ we derive the condition 
	$$\min\left\{\omega^{(6)}_{35},\omega^{(6)}_{14}+\omega^{(6)}_{24} \right\} = 0,$$
	which becomes
	$$\min\left\{c_{135}+c_{235}+c_{356},c_{146} +c_{124}+c_{246}\right\} = 0.$$
\end{example}

We conclude this section with the key Corollary \ref{cor: regular positroidal subdivision blades}, which brings together our work thus far: it reinterprets the usual mechanism, by which a (positive) tropical Plucker vector induces a regular positroidal subdivision by \textit{projecting downward} the folds in the surface induces by a height function, in terms of the blade complex.  It says that this usual method induces a subdivision on the second hypersimplicial boundaries which can be obtained directly using the boundary map and computing the support to find the blade arrangement!  Indeed, thanks to Lemma \ref{lem:posDressToBlades} we know that all coefficients are nonnegative and the support on each face defines a blade arrangement, and in turn a positroidal subdivision.  Let us be more precise.

	\begin{defn}\label{defn: positive tropical Plucker}
	A point $\pi \in \mathbb{R}^{\binom{n}{k}}$ is a positive tropical Plucker vector if it satisfies the positive tropical Plucker relations
	$$\pi_{Jac} + \pi_{Jbd} = \min\{\pi_{Jad} + \pi_{Jbc}, \pi_{Jab} + \pi_{Jcd}\}$$
	for each cyclic order $a<b<c<d$ and each subset $J \subset \binom{\lbrack n\rbrack \setminus \{a,b,c,d\}}{k-2}$.  The positive tropical Grassmannian $\text{Trop}^+G(k,n)$ is the set of all positive tropical Plucker vectors.
\end{defn}

Supposing that $\pi = (\pi_J) \in \mathbb{R}^{\binom{n}{k}}$ satisfies the positive tropical Plucker relations, then it induces a (regular) positroidal subdivision of $\Delta_{k,n}$, by projecting down the folds in the continuous piecewise linear surface over $\Delta_{k,n}$.  Then a positroidal subdivision is induced on each second hypersimplicial faces $\partial_L(\Delta_{k,n})$ with maximal cells $\Pi^{(L)}_1,\ldots, \Pi^{(L)}_M$, say; but any (positroidal) subdivision of the second hypersimplex $\Delta_{2,m}$ is the common refinement of some number of 2-splits; in \cite{Early19WeakSeparationMatroidSubdivision} it was shown that the 2-splits of $\Delta_{2,m}$ are induced exactly by the $\binom{m}{2}-m$  arrangements of a single blade $((1,2,\ldots, m))$ on its nonfrozen vertices.  Thus, the (regular) positroidal subdivisions of $\Delta_{2,m}$ are in bijection with the weakly separated collections of (nonfrozen) 2-element subsets of $\{1,\ldots, n\}$.

Now define 
$$\beta(\mathbf{c}) = \sum_{e_J\in\Delta_{k,n}} \pi_J \mathcal{L}_J = \sum_{e_J\in\Delta_{k,n} \text{ nonfrozen}} c_J \beta_J,$$
after expanding the $\mathcal{L}_J$ in terms of blades $\beta$.

Summarizing the above discussion we finally obtain Corollary \ref{cor: regular positroidal subdivision blades}.
\begin{cor}\label{cor: regular positroidal subdivision blades}
	For each second hypersimplicial face $\partial_L(\Delta_{k,n})$ with $L\in \binom{\lbrack n\rbrack }{k-2}$, we have that
	$$\text{supp}\left(\partial_L(\beta(\mathbf{c}))\right) = \left\{\beta^{(L)}_{i_1j_1},\ldots, \beta^{(L)}_{i_Nj_N}\right\},$$
	where $\{\beta^{(L)}_{i_1,j_1},\ldots, \beta^{(L)}_{i_Nj_N}\}$ is the blade arrangement that induces the same (regular) positroidal subdivision of $\partial_L(\Delta_{k,n})$ that was induced by the height function $p$, i.e., with maximal cells 
	$$\Pi^{(L)}_1,\ldots, \Pi^{(L)}_M.$$
\end{cor}

\begin{proof}
	The proof here is entirely geometric and the argument is similar to that used throughout \cite{Early19WeakSeparationMatroidSubdivision}.
	
	Indeed, let us first recall from \cite{Early19WeakSeparationMatroidSubdivision}, that the intersection of a blade with a hypersimplex $\beta_J\cap \Delta_{k,n}$ coincides with (the inner facets of) an $\ell$-split positroidal subdivision, where $\ell$ is the number of cyclic intervals of the set $J$; but these inner facets are (the images of) the bends of the continuous piecewise linear surface over $\Delta_{k,n}$ when projected down.  So the result holds when $\beta(\mathbf{c}) = \beta_J$ is a single blade with corresponding height function $\mathfrak{h}_J$ (defined modulo lineality).
	
	Now given a linear combination $\beta(\mathbf{c}) \in \mathcal{Z}_{k,n}$ we define a height function over each second hypersimplicial face $\partial_L(\Delta_{k,n})$.  Namely, let
	$$\mathfrak{h}^{(L)}(\mathbf{c}) = (d^{(L)}_{ij})\in \mathbb{R}^{\binom{n-(k-2)}{2}},$$
	where the coefficients $d^{(L)}_{ij}$ are defined by 
	$$ \partial_L\left(\beta(\mathbf{c})\right) = \sum_{i,j} d^{(L)}_{ij}\mathcal{L}^{(L)}_{ij}.$$
	The height function $\mathfrak{h}^{(L)}(\mathbf{c})$ here induces a regular subdivision on $\partial_L(\Delta_{k,n})\simeq \Delta_{2,n-(k-2)}$ in the usual way; now, expanding the $\mathcal{L}^{(L)}_{ij}$ and computing the (positive) coefficients of the blades $\beta^{(L)}_{ij}$ defines a matroidal blade arrangement, where each intersection $\beta^{(L)}_{ij}\cap \partial_L(\Delta_{k,n})$ is then the image of a fold in the surface defined by the height function $\mathfrak{h}^{(L)}(\mathbf{c})$.
\end{proof}

\begin{example}
	Following the construction of Speyer-Williams \cite{SpeyerWilliams2003}, let $\pi = \sum_J \pi_Je^J$ be the tropical Plucker vector where each $\pi_J$ is the piecewise-linear function on $\mathbb{R}^{(k-1)(n-k-1)}$ obtained by tropicalizing the positive (web) parametrization of the totally positive Grassmannian $G_+(k,n)$.  Then the replacement
	$$\pi\mapsto \sum_J \pi_J\mathcal{L}_J$$
	induces a parametrization of $\mathcal{Z}_{k,n}$.
\end{example}
This was also used to parametrize, equivalently, the totally positive tropical Grassmannian, see for instance \cite{arkani2020positive,Brodsky2017,SpeyerWilliams2020} as well as  \cite{AHL2019Stringy,Drummond2019A,HeRenZhang2020,HenkePapa2019}.  However, one of the most significant novel features here is the additional structure: subdivisions are induced not by an auxiliary process, projecting down the lower envelope of the convex hull of the lifts of vertices of the hypersimplex, but by an extremely simple construction: attaching weights to the internal facets of subdivisions, via a distinguished set: the weighted blade arrangements!

\subsection{Coarsest Subdivisions and Planar Kinematic Invariants}
Given $J \in \mathbb{R}^{\binom{n}{k}}$, then $\mathfrak{h}_J \in \mathbb{R}^{\binom{n}{k}}$ can be used to define a surface over the hypersimplex $\Delta_{k,n}$, where the vertex $e_I\in \Delta_{k,n}$ is lifted to a height $\mathfrak{h}_J(e_I)$.  According to the standard prescription, by taking the convex hull of these heights and projecting down the lower envelope one obtains a subdivision of $\Delta_{k,n}$.  Equivalently, one may simply calculate the regions of linearity of the function $\rho_J(x)$.

For any $J \in \binom{\lbrack n\rbrack}{k}$, the function $\rho_J$ is linear over a collection of matroid polytopes, with vertices the so-called Schubert matroids.  This collection defines a subdivision of $\Delta_{k,n}$ which is a multi-split, see Definition \ref{defn:multisplit}.  In particular, it is coarsest \cite{Schroeter}.

\subsection{Negativity in $\mathcal{Y}_{3,n}$}
We conclude our general discussion of the blades complex with a result concerning which coefficients can be negative for an element $\beta(\mathbf{c}) \in \mathcal{Y}_{3,n}$, so in particular for $\mathcal{Z}_{3,n}$ the result holds true.

Namely, we show that only the \textit{totally} non-frozen $\beta_{u,v,w}$'s can be negative in a linear combination $\beta(\mathbf{c}) = \sum_{u<v<w} c_{uvw}\beta_{uvw} \in\mathcal{Y}_{3,n}$.
\begin{prop}\label{lem: no negative 2-splits}
	Suppose that $\beta(\mathbf{c}) = \sum_{r<s<t} c_{rst}\beta_{rst} \in\mathcal{Y}_{3,n}$ with coefficients $c_{rst}\in\mathbb{Z}$, where the sum is over all nonfrozen 3-element subsets $\{r,s,t\}$ of $\Delta_{3,n}$.  Then $c_{i,i+1,j},c_{i,j,j+1} \ge 0$ for all $1\le i<j\le n$.
\end{prop}
\begin{proof}
	Supposing that the claim fails, without loss of generality let us suppose that $c_{i,i+1,j}<0$.  Then since all coefficients in $\partial(\beta(\mathbf{c}))$ must be nonnegative, we must have that $\partial(\beta_{i,i+1,j})$ cancels completely on the boundary with other nonzero terms in $\beta(\mathbf{c})$.  So let us determine which $\beta_{a,b,c}$'s could cancel with it in the sum $\beta(\mathbf{c})$.  We have
	\begin{eqnarray*}
		\partial(\beta_{i,i+1,j}) & = &  \sum_{a=j+1}^{i-1} \beta^{(a)}_{i+1,j} + \beta^{(i)}_{i+1,j} + \beta^{(i+1)}_{i,j} + \sum_{a=i+2}^j \beta^{(a)}_{i,i+1} \\
		& = &  \sum_{a=j+1}^i \beta^{(a)}_{i+1,j} + \beta^{(i)}_{i+1,j} + \beta^{(i+1)}_{i,j}.
	\end{eqnarray*}
	Notice that there is a unique $\beta_{a,b,c}$ such that $\partial_i(\beta_{a,b,c}) = \beta^{(i)}_{i+1,j}$, namely $\beta_{i,i+1,j}$ itself; but this was already assumed to be negative, which yields a contradiction.  
\end{proof}
\newpage
\section{Some building blocks for weighted blade arrangements in $\mathcal{Z}_{k,n}$}\label{sec:building blocks}

Fix the usual cyclic order $\alpha = (12\cdots n)$.

In this section we move beyond basic theory to outline some basic techniques and a construction for $\mathcal{Z}_{k,n}$ which are not seen and are not apparent in the usual realization of the positive tropical Grassmannian.  A deeper study will be conducted in \cite{EarlyPrep}, digging into the physical structure of the generalized biadjoint scalar $m^{(k)}(\alpha,\alpha)$.  Unless otherwise stated, in this section we will assume that $3\le k\le n-3$ and any $k$-element subset $J$ of $\{1,\ldots, n\}$ consists of at least three cyclic intervals.

Given a vertex $e_J\in\Delta_{k,n}$, then it decomposes as $J = J_1\cup \cdots \cup J_\ell$ into intervals, with $\ell\ge 3$, that are cyclic with respect to $\alpha$, where $\min\{J\}\in J_1$, say; let $C_1,\ldots, C_\ell$ be the interlaced complements, so that their concatenation recovers $\alpha$:
$$\alpha = (C_1,J_1,C_2,J_2\ldots, C_\ell,J_\ell).$$
We construct a matroidal weighted blade arrangement with source point $e_J$, for each vertex of the Cartesian product
$$\mathcal{D}_J(\Delta_{k,n}) := \left(\Delta_{\vert J_1\vert,J_1\cup C_2}\setminus \{e_{J_1}\}\right)\times \left(\Delta_{\vert J_2\vert,J_2\cup C_3}\setminus \{e_{J_2}\}\right) \times \cdots \times \left(\Delta_{\vert J_\ell\vert,J_\ell\cup C_1}\setminus\{ e_{J_\ell}\}\right),$$
where each factor involves a hypersimplex 
$$\Delta_{\vert J_j\vert ,J_j\cup C_{j+1}} = \left\{\sum_{i\in J_j \cup C_{j+1}} x_{i}e_i \in \Delta_{k,n}: x_{J_j\cup C_{j+1}} = \vert J_j\vert  \right\}$$
from which is deleted the vertex $e_{J_j}$.

For instance if $\Delta_{4,9}$ and $J=\{2,5,7,8\}$ then for the cyclic order we find 
$$(1,2,3,4,5,6,7,8,9) = (C_1,J_1,C_2,J_2,C_3,J_3)=(\{9,1\},\{2\},\{3,4\},\{5\},\{6\},\{7,8\}),$$
and the $\mathcal{D}_J$ is the Cartesian product of a triangle, an octahedron and a line segment (with one vertex deleted from each)
$$\mathcal{D}_J = (\Delta_{1,\{2,3,4\}}\setminus \{e_2\}) \times (\Delta_{2,\{5,6,7,8\}}\setminus \{e_{56}\})\times (\Delta_{1,\{9,1\}}\setminus \{e_9\}),$$
so there are $(3-1)\cdot \left(\binom{4}{2}-1\right)\cdot (2-1) =10$ weighted blade arrangements associated to this $J$.

Let $J\in\Delta_{k,n}$ with cyclic intervals $J_1,\ldots, J_\ell$, labeled with the convention that $\min(J) \in J_1$, say.  Given any vertices
$$e_{I_1}\in  \left(\Delta_{\vert J_1\vert,J_1\cup C_{2}}\setminus \{e_{J_1}\}\right),\ e_{I_2}\in  \left(\Delta_{\vert J_2\vert,J_2\cup C_{3}}\setminus \{e_{J_2}\}\right),\ldots, e_{I_\ell}\in  \left(\Delta_{\vert J_\ell\vert,J_\ell\cup C_{1}}\setminus \{e_{J_\ell}\}\right),$$
define $I=I_1\cup\cdots \cup I_\ell$.  
Then set
\begin{eqnarray*}
	\tau_{e_J,e_I} = -(\ell-2)\beta_J + \sum_{j=1}^\ell \beta_{J_1\cup J_2 \cup \cdots \cup I_j\cup \cdots J_j},
\end{eqnarray*}
where in the $j^\text{th}$ summand, $J_j$ has been replaced with $I_j$.  We emphasize that we must have $e_I \in \mathcal{D}_J$.

By counting the numbers of vertices in the Cartesian product it follows that there are
$$\prod_{j=1}^\ell \left(\binom{\vert J_j\vert + \vert C_{j+1}\vert }{\vert J_j\vert}-1\right)$$
weighted blade arrangements of this form, with negative blade $-(\ell-2)\beta_J$.

\begin{rem}
	For $k\ge 4$, consider the set of all totally nonfrozen vertices $(\Delta_{k,n})^{\circ}$, say.  It is easy to check that the corresponding set of weighted blade arrangements is minimally closed with respect to the boundary operator.  This means that for any $e_J\in (\Delta_{k,n})^{\circ}$ and any $e_I \in \mathcal{D}_{J}$, then 
	$$\partial_j(\tau_{e_J,e_I}) = \tau_{e_{J'},e_{I'}}$$
	where $e_{J'} \in (\partial_j(\Delta_{k,n}))^\circ$ and $e_{I'} \in \mathcal{D}_{J'}(\partial_j(\Delta_{k,n}))$.
	
	In words, the restriction of the hierarchy of weighted blades arrangements to totally nonfrozen vertices $e_J$ is minimally closed with respect to the boundary operators $\partial_j$: any such $\tau_{e_J,e_I}$ induces a smaller $\tau_{e_{I'},e_{J'}}$ on each face of $\Delta_{k,n}$, where $e_{J'}$ is again totally nonfrozen.
	
\end{rem}

Let us compute several of the boundaries to see that these families of elements $\tau_{e_J,e_I}$ are all related.  

\begin{example}
	The element $\tau_{e_{1,3,5,7,9},e_{2,4,6,8,10}} \in\mathcal{Z}_{5,12}$ has the form 
	$$\tau_{e_{1,3,5,7,9},e_{2,4,6,8,10}} = -3\beta_{1,3,5,7,9} + \beta_{2,3,5,7,9} + \beta_{1,4,5,7,9} + \beta_{1,3,6,7,9} + \beta_{1,3,5,8,9} + \beta_{1,3,5,7,10}.$$
	We find for instance
	\begin{eqnarray*}
		\partial_1(\tau_{e_{1,3,5,7,9},e_{2,4,6,8,10}}) & = & -2 \beta^{(1)} _{3,5,7,9}+\beta^{(1)} _{3,5,7,10}+\beta^{(1)} _{3,5,8,9}+\beta^{(1)} _{3,6,7,9}+\beta^{(1)} _{4,5,7,9}\\
		\partial_{6,11}(\tau_{e_{1,3,5,7,9},e_{2,4,6,8,10}}) & = & -\beta^{(6,11)}  _{3,5,9}+\beta^{(6,11)}  _{3,5,10}+\beta^{(6,11)}  _{3,7,9}+\beta^{(6,11)} _{4,5,9}\\
		\partial_{4,7,11}(\tau_{e_{1,3,5,7,9},e_{2,4,6,8,10}}) & = & \beta^{(4,7,11)}_{3,10} + \beta^{(4,7,11)}_{5,9}.
	\end{eqnarray*}
	On every face of $\Delta_{5,12}$ we get a lower order element in the same family!
	
	We leave as an exercise, as desired, to repeat the computation for the element 
	$$\tau_{e_{1,3,6,7,11},e_{2,4,8,9,12}} = -2\beta_{1,3,6,7,11} + \beta_{2,3,6,7,11} + \beta_{1,4,8,9,11} + \beta_{1,3,8,9,11} + \beta_{1,3,6,7,12} \in \mathcal{Z}_{5,14}.$$
\end{example}

\begin{rem}
	An weighted blade arrangement $\tau_{e_J,e_I}$ for $e_I,e_J\in\Delta_{k,n}$ as above, appears for the first time at $n=2k$, in $\Delta_{k,2k}$.  In this case the elements $a_i$ are separated by cyclic gaps of 2 and each $r_i=1$.
\end{rem}

\begin{conjecture}
	Each element $\tau_{J,I} \in \mathcal{Z}_{k,n}$ induces a coarsest positroidal subdivision of $\Delta_{k,n}$; it generates a ray of the positive tropical Grassmannian $\text{Trop}^+G(k,n)$.
\end{conjecture}
We will return to this result in \cite{EarlyPrep}.

We conclude with an application of Theorem \ref{thm:tropGrass onto Zkn} to give with the full classification of rays of the positive tropical Grassmannian $\text{Trop}^+G(3,n)$, for $n\ge 9$, as matroidal weighted blade arrangements, up to labeling.  See Figure \ref{fig:generalized-three-coordinate-permutohedra-biadjoint-scalar-39}.

\section*{Acknowledgements}
We would like to thank Freddy Cachazo for stimulating discussions, collaborations and encouragement, providing physical motivation for this work.  We also thank Nima Arkani-Hamed, James Drummond, Alfredo Guevara, Chrysostomos Kalousios, Thomas Lam, Tomasz Lukowski, William Norledge, Matteo Parisi, Alex Postnikov, Benjamin Schroeter, Marcus Spradlin, and Yong Zhang for helpful discussions and comments.

This research was supported in part by a grant from the Gluskin Sheff/Onex Freeman Dyson Chair in Theoretical Physics and by Perimeter Institute. Research at Perimeter Institute is supported in part by the Government of Canada through the Department of Innovation, Science and Economic Development Canada and by the Province of Ontario through the Ministry of Colleges and Universities.

\appendix

\section{Physical Motivation}\label{sec:physical motivation}

In this section, for sake of completeness we outline the motivating physical structures which have been the subject of intensive study since their introduction in \cite{CEGM2019}: the original motivation for this paper was to classify poles of the $n$-point generalized biadjoint amplitude $m^{(3)}(\alpha,\alpha)$.  

\subsection{Configuration space and kinematic space}

Let 
$$\text{Conf}\left(\mathbb{C}^k,n\right) = \left\{g\in(\mathbb{CP}^k)^n: \det\left(g_{j_1},\ldots, g_{j_k}\right)\not=0\text{ for all } J\in\binom{\lbrack n\rbrack}{k} \right\}$$ 
be the configuration space of $n$ generic points in complex projective space $\mathbb{CP}^{k-1}$, so in particular, any $k+1$ points determine a projective frame.

Following the construction and nomenclature from \cite{CEGM2019}, the starting point is the \textit{potential function} $\mathcal{P}'_{k,n}:\text{Conf}\left(\mathbb{C}^{k},n\right)\times \mathbb{R}^{\binom{n}{k}} \rightarrow \mathbb{C},$
\begin{eqnarray}
(g,s) \mapsto \sum_{J\in\binom{\lbrack n\rbrack}{k}} s_{j_1\cdots j_k}\log\left(\det(g_{j_1},\ldots, g_{j_k})\right).
\end{eqnarray}

The \textit{kinematic space} $\mathcal{K}_{k,n}$  is the subspace of $\mathbb{R}^{\binom{n}{k}}$ that is cut out by the $n$ linearly independent equations 
$$ \sum_{J\in\binom{\lbrack n\rbrack}{k}:\ J\ni a} s_{J}=0 \text{ for each } a=1,\ldots, n.$$
By scaling each of the $n$ points $g_1,\ldots, g_n\in\mathbb{C}^k$ to check for projective invariance of the potential function, it can be seen that the kinematic space $\mathcal{K}_{k,n}$ is the largest subspace of $\mathbb{R}^{\binom{n}{k}}$ on which $\mathcal{P}'_{k,n}$ passes to a well-defined map on the whole quotient:
$$\mathcal{P}_{k,n}:\left(GL(k)\backslash\text{Conf}\left(\mathbb{CP}^{k-1},n\right)\right)\times \mathcal{K}_{k,n} \rightarrow \mathbb{C}.$$

\subsection{Scattering Equations}

Fixing a projective frame on $\mathbb{CP}^{k-1}$, say the standard one
$$(g_1,\ldots, g_{k+1})=(e_1,e_2,\ldots, e_k,e_1+\cdots+e_k),$$
the scattering equations are defined with respect to inhomogeneous coordinates
$$\{z_{i,j}:i=1,\ldots, k-1,\ j = k+2,\ldots, n\}.$$
The affine chart then becomes
\begin{eqnarray}\label{eq:affine chart}
\mathcal{C}=\left\{\begin{bmatrix}
1 & 0 & \cdots  & 0 & 0 & 1 & z_{1,k+2} &   & z_{1n} \\
0 & 1 &  &  & \vdots  & 1 & z_{2,k+2} & & z_{2n}\\
\vdots  &  & \ddots &  &  & \vdots  & \vdots &\cdots  & \\
0 & \cdots  &  & 1 & 0  & 1 & z_{k-1,k+2} &  & z_{k-1,n}\\
0 & 0 & \cdots & 0 & 1 & 1 & 1 &  & 1
\end{bmatrix}: z_{ij}\in\mathbb{C}\right\}.
\end{eqnarray}
\begin{defn}
	With respect to the given chart $\mathcal{C}$, the scattering equations are the $(k-1)(n-k-1)$ equations
	$$\left\{\frac{\partial \mathcal{P}_{k,n}}{\partial z_{ij}}: 1\le i \le k-1,\ k+2\le j\le n \right\},$$
	where $z_{ij}$ are the inhomogeneous coordinates on the chart $\mathcal{C}$.
\end{defn}
\begin{rem}
	Of course, one can show that solutions to the scattering equations are actually independent of the chart $\mathcal{C}$.  
\end{rem}

\subsection{Jacobian Matrix}

In the computation of the generalized biadjoint scalar $m^{(k)}(\alpha,\alpha)$, one finds the so-called reduced Jacobian determinant $\det'\Phi^{(k)}$ of the gradient of the potential function $\mathcal{P}_{k,n}$.  First define the Hessian matrix 
$$\Phi^{(k)} = \left(\frac{\partial^2 \mathcal{P}_{k,n}}{\partial w_{ab}\partial w_{cd}}\right)_{(a,b),(c,d)\in \{1,\ldots, k-1\}\times \{1,\ldots, n\}}.$$
Here $(w_{1j},w_{2j},\ldots, w_{k-1,j},1)$ is an affine chart on the $j^{th}$ copy of $\mathbb{CP}^{k-1}$.

Denote by 
$$V_{i_1\cdots i_{k+1}} := \prod_{j=1}^{k+1}\det(g_{i_1},\ldots \widehat{g_{i_j}}\cdots g_{i_{k+1}}),$$
which is a natural $(k+1)\times(k+1)$ generalization of the usual $3\times 3$ Vandermonde determinant in the case $k=2$,  here for any $(k+1)$-element subset $\{i_1,\ldots, i_{k+1}\} \in \binom{\lbrack n\rbrack}{k+1}$ of $\{1,\ldots, n\}$.

Following \cite{CEGM2019}, choose any pair of $(k+1)$-element subsets $A, B$ of the columns $\{1,\ldots, n\}$, let 
$$\lbrack \Phi^{(k)}\rbrack^A_{B} = \left(\frac{\partial^2 \mathcal{P}_{k,n}}{\partial w_{a_1a_2}\partial w_{b_1b_2}}\right)^{(a_1,a_2)\in\lbrack k-1\rbrack \times(\lbrack n\rbrack \setminus A)}_{(b_1,b_2)\in\lbrack k-1\rbrack \times(\lbrack n\rbrack \setminus B)}$$
be the \textit{reduced} Hessian matrix of $\mathcal{P}_{k,n}$, which can be obtained directly from the Hessian matrix $\Phi^{(k)}$ by deleting from it all rows labeling points in $A$ and all columns labeling points in $B$.

\begin{claim}
	The ratio 
	$$\frac{\det\left(\lbrack \Phi^{(k)}\rbrack^A_{B}\right)}{V_A V_B}$$
	is independent of the choice of subsets $A,B \subset \lbrack n\rbrack$.  
	
	Therefore one defines
	$${\det}'\left( \Phi^{(k)}\right) = \frac{\det\left(\lbrack \Phi^{(k)}\rbrack^A_{B}\right)}{V_A V_B},$$
	where it is understood that the choice of $(k+1)$-element subsets $A,B \subset  \lbrack n\rbrack$ is fixed (arbitrarily) at the outset
\end{claim}

\subsection{Generalized Biadjoint Scalar, planar basis and blades}\label{subsec:generalizedBiadjointScalar}

The so-called generalized Parke-Taylor factor is the inverse of the product of maximal $k\times k$ minors $d_{i_1\cdots i_k} = \det(g_{i_1},\ldots, g_{i_k})$,
$$PT^{(k)}\left(\alpha\right) = \frac{1}{d_{\alpha_1\alpha_2\cdots \alpha_k}d_{\alpha_2\alpha_3\cdots \alpha_{(k+1)}}\cdots d_{\alpha_n\alpha_1\cdots \alpha_{(k-1)}}},$$
where $\alpha = (\alpha_1\cdots \alpha_n)$ is a cyclic order on $\{1,\ldots, n\}$. 

\begin{defn}
	For any pair of cyclic orders $\alpha = (\alpha_1\cdots\alpha_n)$ and $\beta = (\beta_1\cdots \beta_n)$ on $\lbrack n\rbrack = \{1,\ldots, n\}$, define a function $m^{(k)}(\alpha,\beta)$, by 
	\begin{eqnarray*}
		m^{(k)}(\alpha,\beta) & = & \sum_{\text{soln}}\left(\frac{1}{\det'\left(\Phi^{(k)}\right)}PT^{(k)}\left(\alpha\right)PT^{(k)}\left(\beta\right)\bigg\vert_{\text{soln}}\right),
	\end{eqnarray*}
	Here the sum is over all solutions to the scattering equations.  It is an open problem to show in general that the number of solutions is finite for generic choices of kinematics, values of the Mandelstam variables $(\mathbf{s})$; however, see \cite{CachazoSingularSolutions} for enumeration of singular solutions.  However, in \cite{CE2020} Cachazo and the author proved that for all possible $k$ and $n$, on a particular $(n-2)$ dimensional subspace of $\mathcal{K}_{k,n}$, called there \textit{minimal kinematics}, the scattering equations possess a \textit{unique} solution which is in the image of a Veronese embedding!

\end{defn}

\begin{example}
	For sake of comparison, let us express in the planar basis some of the directly computed values from \cite{CEGM2019} of the $n=6$ point $m^{(3)}(\alpha,\beta)$:
	\begin{eqnarray}\label{eqn: biadjoint amplitude Example}
	m^{(3)}((1,2,3,4,5,6),(1,2,6,4,3,5)) & = & \frac{1}{\eta_{125}\eta_{245}\eta_{256}\eta_{124}}\nonumber\\
	m^{(3)}((1,2,3,4,5,6),(1,2,5,4,6,3)) & = & \frac{1}{\eta_{236}\eta_{356}\eta_{235}\eta_{256}}\nonumber\\
	m^{(3)}((1,2,3,4,5,6),(1,2,5,4,3,6)) & = & \frac{1}{\eta_{245}\eta_{125}}\left(\frac{1}{\eta_{145}}+\frac{1}{\eta_{256}}\right)\left(\frac{1}{\eta_{124}} + \frac{1}{\eta_{235}}\right)\nonumber\\
	m^{(3)}((1,2,3,4,5,6),(1,2,6,4,5,3)) & = & -\frac{1}{\eta_{256}\eta_{235}}\left(\frac{1}{\eta_{245}}+\frac{1}{\eta_{356}}\right)\left(\frac{1}{\eta_{125}} + \frac{1}{\eta_{236}}\right)\\
	m^{(3)}((1,2,3,4,5,6),(1,4,5,6,2,3)) & = & \frac{1}{\eta_{136}\eta_{236}\eta_{356}}\left(\frac{1}{\eta_{235}}+\frac{1}{\eta_{346}}\right)\nonumber \\
	m^{(3)}((1,2,3,4,5,6),(1,2,5,6,3,4)) & = & \frac{\eta_{346} + \eta_{256} + \eta_{124}}{\eta_{346}\eta_{256}\eta_{124}\eta_{246}(-\eta_{246} + \eta_{124}+\eta_{346}+\eta_{256})}\nonumber
	\end{eqnarray}
	Note the presence of $-\eta_{246} + \eta_{124}+\eta_{346}+\eta_{256}$; this is the $k=3$ case of the construction in Section \ref{sec:building blocks}, upon the substitution $\beta_J\mapsto  \eta_J$!  Here $\beta_{246}$ corresponds to $((12_1 34_1 56_1))$, while $-\beta_{246} + \beta_{124}+\beta_{346}+\beta_{256}$ corresponds to $((12_1 56_1 34_1))$.  These induce two out of the four positroidal three splits of $\Delta_{3,6}$.  Note that we have now recovered the bipyramidal relation first noticed by \cite{SpeyerWilliams2003}!  In the blade basis it appears to be an almost trivial cancellation: 
	$$\beta_{246} + (-\beta_{246} + \beta_{124} + \beta_{346} + \beta_{256}) =  \beta_{124} + \beta_{346} + \beta_{256},$$
	but it in fact expresses the nontrivial fact that the corresponding sets of positroidal subdivisions have the same common refinement.  In terms of hypersimplex blades we have that the superpositions of blades, respectively
	$$\left\{((12_1 34_1 56_1)), ((12_1 56_1 34_1))\right\}$$
	and
	$$\left\{((12_1 3456_2)), ((1234_2 56_1)), ((1256_2 34_1))\right\},$$
	induce the same (finest) positroidal subdivision of $\Delta_{3,6}$!  This identity is reflected in the nontrivial numerator in the last line of Equation \eqref{eqn: biadjoint amplitude Example}.
\end{example}

\begin{figure}[h!]
	\centering
	\includegraphics[width=1\linewidth]{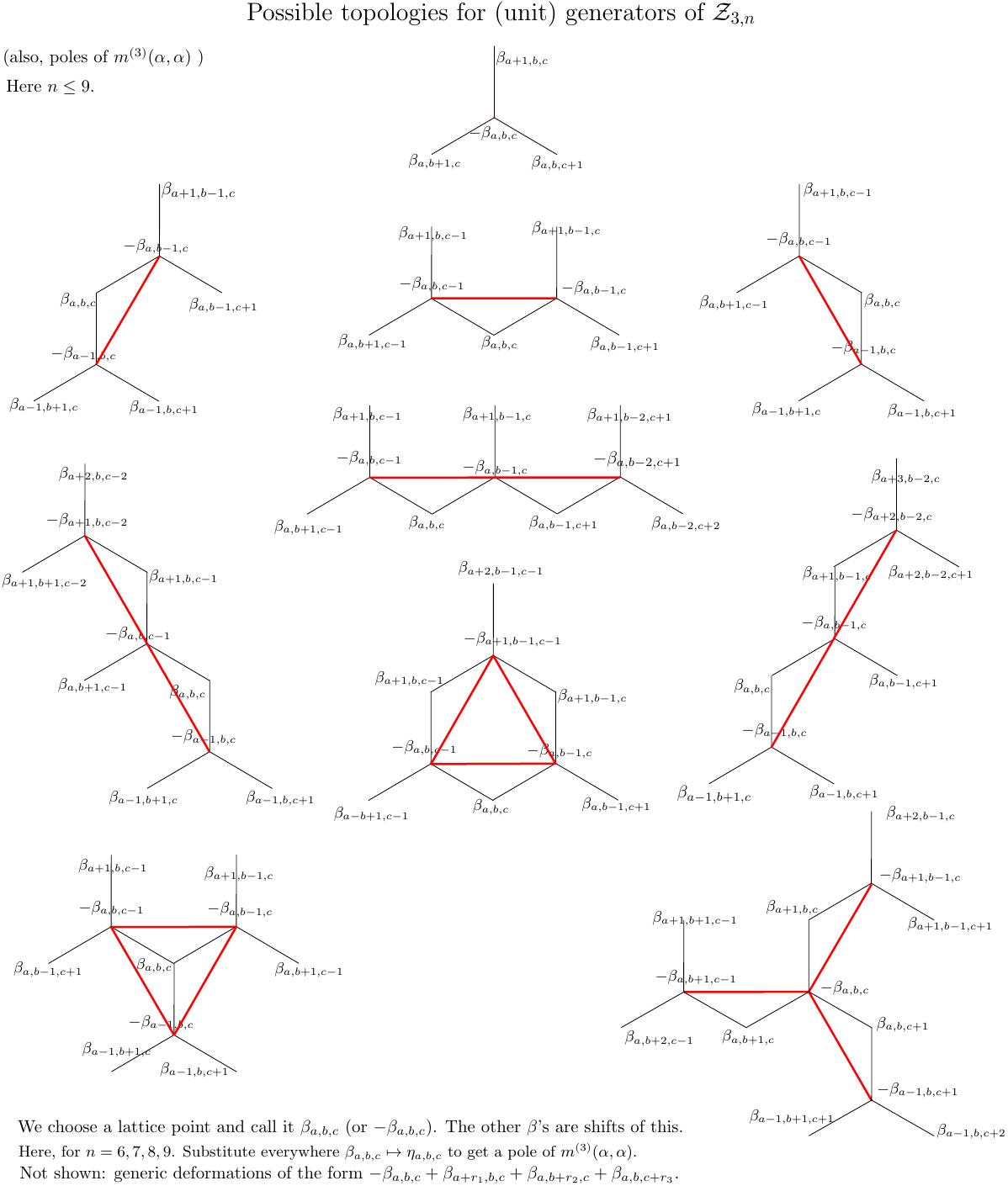}
	\caption{Classification (up to labeling) of all (coarsest) weighted blade arrangements of $\mathcal{Z}_{3,n}$ (= rays of $\text{Trop}_+G(3,n)$, by Lemma \ref{lem:posDressToBlades}) for $n\le 9$. Translated from the sets of tropical Plucker vectors for rays of the positive tropical Grassmannian as computed in \cite{CGUZ2019,Drummond2019A,HeRenZhang2020}.  The top entry first appears at $n=6$.  Entries with two tripods appear first at $n=8$.  Entries with three tripods appear first at $n=9$.  The red lines are edges on the root lattice.  Here a (unit) tripod is a weighted blade arrangement $-\beta_{a,b,c} + \beta_{a+r_1,b,c} +\beta_{a,b+r_2,c} +\beta_{a,b,c+r_3}$ where all $r_i=1$.}
	\label{fig:generalized-three-coordinate-permutohedra-biadjoint-scalar-39}
\end{figure}

\end{document}